\documentclass[review,onefignum,onetabnum]{siamart190516}
\allowdisplaybreaks

\usepackage{amsmath,amssymb,color}

\usepackage{graphicx}
\usepackage{subfigure}
\usepackage{subdepth}
\usepackage{lineno,etoolbox}

\allowdisplaybreaks
\DeclareGraphicsExtensions{.pdf,.eps}

\newcommand{\bu}{\boldsymbol u}

\newcommand{\bv}{\boldsymbol v}
\newcommand{\bV}{\boldsymbol V}
\newcommand{\bw}{\boldsymbol w}
\newcommand{\bbeta}{\boldsymbol  \eta}

\newcommand{\bz}{\boldsymbol z}

\newcommand{\bx}{\boldsymbol x}
\newcommand{\bn}{\boldsymbol n}
\newcommand{\by}{\boldsymbol y}

\newcommand{\be}{\boldsymbol e}

\newcommand{\bvar}{\boldsymbol \varphi}

\newcommand{\bU}{\boldsymbol U}

\newcommand{\bs}{\boldsymbol s}

\newcommand{\bff}{\boldsymbol f}

\newtheorem{remark}[theorem]{Remark}

\newcommand{\linenomathpatch}[1]{%
  \cspreto{#1}{\linenomath}%
  \cspreto{#1*}{\linenomath}%
  \csappto{end#1}{\endlinenomath}%
  \csappto{end#1*}{\endlinenomath}%
}
\linenomathpatch{equation}
\linenomathpatch{gather}
\linenomathpatch{multline}
\linenomathpatch{align}
\linenomathpatch{alignat}
\linenomathpatch{flalign}

\usepackage{hyperref}

\usepackage{macros}

\title{POD-ROMs for incompressible flows including snapshots of the temporal derivative of the 
full order solution}
\author{Bosco
Garc\'{\i}a-Archilla\thanks{Departamento de Matem\'atica Aplicada
II, Universidad de Sevilla, Sevilla, Spain, (\email{bosco@esi.us.es}), \funding{The research of Bosco
Garc\'{\i}a-Archilla has been supported by
Spanish MCINYU under grants PGC2018-096265-B-I00 and PID2019-104141GB-I00.}}
\and
Volker John\thanks{Weierstrass Institute for Applied Analysis and Stochastics,
Leibniz Institute in Forschungsverbund Berlin e. V. (WIAS), Mohrenstr. 39, 10117 Berlin, Germany.
Freie Universit\"at of Berlin,
Department of Mathematics and Computer Science,
Arnimallee 6, 14195 Berlin, Germany, (\email{john@wias-berlin.de}).}
  \and Julia Novo\thanks{Departamento de
Matem\'aticas, Universidad Aut\'onoma de Madrid, Spain,  (\email{julia.novo@uam.es}), 
\funding{The research of Julia Novo has been supported by Spanish MINECO
under grants PID2019-104141GB-I00 and VA169P20.}}}

\begin{document}
\nolinenumbers

\maketitle

\begin{abstract}
 In this paper we study the influence of including snapshots that approach the velocity time derivative in the numerical
 approximation of the incompressible Navier--Stokes equations by means of proper orthogonal decomposition (POD) methods.  
 Our set of snapshots includes the velocity approximation at the initial time from a full order mixed finite element
 method (FOM) together with approximations to the time derivative at different times. The approximation 
 at the initial velocity can be replaced by the mean value of the velocities at the different times so 
 that implementing the method to the fluctuations, as done mostly in practice, only approximations
 to the time derivatives are included in the set of snapshots. For the POD method we
 study the differences between projecting onto $L^2$ and $H^1$. In both cases, pointwise in time error bounds
 are proved. Including grad-div stabilization both in the FOM and POD methods, error bounds with constants independent of inverse powers of the viscosity are obtained. 
 \end{abstract}

\begin{keyword}
incompressible Navier--Stokes equations; proper orthogonal decomposition (POD);
reduced order models (ROMs); snapshots of the temporal derivative of the full order solution; grad-div stabilization; robust pointwise in time estimates
\end{keyword}

\begin{AMS}
65M12, 65M15, 65M60
\end{AMS}

\section{Introduction}

It is well known that the computational cost of direct numerical simulations, also called full order methods (FOMs), can be reduced by using reduced order models (ROMs). 
In this paper, we study ROMs based  on proper orthogonal decomposition (POD) methods, so-called POD-ROMs. 
The computation of the reduced basis uses solutions of a FOM, so-called snapshots.

We study incompressible flow problems that are modeled by means of the 
incompressible Navier--Stokes equations
\begin{equation}
\begin{array}{rcll}
\label{NS} \bu_t -\nu \Delta \bu + (\bu\cdot\nabla)\bu + \nabla p &=& \bff &\text{in }\ (0,T]\times\Omega, \\
\nabla \cdot \bu &=&0&\text{in }\ (0,T]\times\Omega,
\end{array}
\end{equation}
in a bounded domain $\Omega \subset {\mathbb R}^d$, $d \in \{2,3\}$, with initial condition $\bu(0)=\bu^0$. In~\eqref{NS},
$\bu$ is the velocity field, $p$ the kinematic pressure, $\nu>0$ the kinematic viscosity coefficient,
 and $\bff$ represents the accelerations due to external body forces acting
on the fluid. The Navier--Stokes equations \eqref{NS} have to be complemented with boundary conditions. For simplicity,
we only consider homogeneous
Dirichlet boundary conditions $\bu = \boldsymbol 0$ on $[0,T]\times \partial \Omega$.

This paper studies the impact of 
including  approximations of the temporal derivative of the velocity in the set of snapshots.
The idea consists in taking, in addition to the  mixed finite element approximation to the velocity at the initial time, $\left\{\bu_h(t_0)\right\}$,
the time derivatives of the mixed finite element approximations, $\left\{\bu_{h,t}(t_j)\right\}$. These temporal derivatives can be easily computed using the right-hand side of the mixed finite element Galerkin equation {(see Remark~\ref{comp_ut}
below)}.
In the present paper we follow an idea presented in the recent paper \cite{locke_singler} in which it is shown
that there is no need to include in the set of snapshots other than one approximation to the velocity at a fixed time, instead
of the full set $\left\{\bu_{h}(t_j)\right\}_{j=0}^M$ as it is usually done in the literature. The numerical analysis in \cite{locke_singler}
is carried out for the heat equation and with difference quotients, $\left\{(\bu_h(t_j)-\bu_h(t_{j-1})/(t_j-t_{j-1})\right\}$, approaching the time derivative. In the present paper
we consider instead the Galerkin time derivatives although the analysis for the difference quotients case is essentially the same (or even simpler). Actually, in practice, any approximation to the time derivative can work equally well. We also prove that the snapshot at the initial value can be replaced by the mean value $\overline\bu_h=(M+1)^{-1}\sum_{j=0}^M\bu_h(t_j)$, which
can be more efficient in the numerical simulations. It is standard to apply the POD method to the fluctuations $\left\{\by_h(t_j)\right\}_{j=0}^M=\left\{\bu_{h}(t_j)-\overline\bu_h\right\}_{j=0}^M$ that have vanishing mean by definition.  Then, with the method
we propose only approximations to the derivatives are needed in this case.

Several works in the literature have already studied
the subject of increasing the set of snapshots with approximations to the time derivatives. 
However, apart from \cite{locke_singler}, all of them include as snapshots the set of the approximations at different times, instead
of only one snapshot. Also, starting with the pioneering paper \cite{kunisch}, most of these papers include a different 
type of approximation than considered in this paper, namely difference quotients. In particular, to the best of our knowledge, 
this is the first paper that studies the inclusion of the temporal derivatives of the
mixed finite element velocity approximations in order to generate  the reduced order basis for the 
incompressible Navier--Stokes equations. With this more general setting we can deduce that in practice
any approximation to the time derivative produces essentially the same results.

The initial motivation for investigating a different approximation than difference quotients in the set of snapshots
is that the results for difference quotients are ambivalent. 
On the one hand, from the theoretical point of view, the inclusion of the difference quotients possesses
some advantages. First of all, it allows to prove optimal error bounds for the POD-ROM when 
the POD basis functions are based on the projection onto the Hilbert space $X=H_0^1(\Omega)^d$, see 
\cite{kunisch,iliescu_et_al_q,singler}. In this way, the standard finite element error analysis 
is mimicked, in which the Ritz or Stokes projection is used to split the error in a projection error 
and a discrete remainder. It was observed that if the POD basis functions
are based on the projection onto the Hilbert space $X=L^2(\Omega)^d$, the difference quotients are not needed to prove optimal error bounds in certain norms, see  
\cite{chapelle,iliescu_et_al_q,singler,novo_rubino}. However, as pointed out in \cite{koc_rubino_et_al}, even in this case, the
inclusion of the difference quotients allows to get pointwise estimates in time that generally 
cannot be proved if there are no difference quotients in the set of snapshots. On the other hand, from the numerical point of view, it is not clear that the difference quotients should be included in the actual 
simulations with the POD-ROM. In fact, it is reported in  \cite{kean_sch,nos_pos_supg} that the POD-ROM 
without the difference quotients performs considerably better than with the difference quotients. 

Trying to keep the theoretical advantages of including approximations of the temporal derivative of the 
velocity in the set of snapshots but relaxing their drawbacks in practical simulations, we study in this paper, both theoretically and numerically, the 
inclusion of time derivatives of the discrete velocity. As in \cite{locke_singler}, our approach has half the number of snapshots as
in the standard POD finite quotient approach. Also, as in the difference quotient case, we are able to get 
pointwise in time estimates using as projecting space $X=H_0^1(\Omega)^d$ and $L^2(\Omega)^d$. 
To this end, we follow \cite[Lemma~3.3]{locke_singler} (see also \cite[Lemma 3.6]{koc_rubino_et_al}) 
to prove that the $X$-norm of a function at any point in time (let's say $t_j$) is bounded in terms of the $X$-norm  of the value
at $t_0$ plus the mean values of its time derivative taken in a time
interval (let's say $\left\{t_0,t_1,\ldots,t_M\right\}$, up to an error that tends to zero with the length
of the time step. From the numerical point of view, including the snapshots of the time derivatives avoids the potential problem of performing badly conditioned operations in the computation of the snapshots like
$\left\{(\bu_h(t_j)-\bu_h(t_{j-1}))/(t_j-t_{j-1})\right\}$ because both numerator and denominator may suffer from numerical cancellation. Finally, we mimic model reduction ideas coming from dynamical systems, 
where using snapshots from the time derivative is a more common approach, see for example \cite{kostova}.

For the error analysis in the present paper, instead of considering a concrete fully discrete scheme from which 
the values $\left\{\bu_h(t_j)\right\}$ are taken, we consider a continuous-in-time method, which has 
some advantages. In practice, one always computes the snapshots with a fully discrete method but the error
analysis based on the continuous-in-time method holds for any time integrator used in the FOM. With this approach 
one can use different time steps for the FOM, which 
produces the snapshots, and for the POD-ROM method. Our error analysis takes into account the temporal error coming from the POD-ROM.
%
{The error analysis of the present paper can be easily adapted to include the case in which the snapshots are 
computed with a fully discrete method. Errors with respect the fully discrete approximations can be expressed in terms of errors with respect the continuous-in-time approximation (estimated in the present paper) and further terms depending on the error of the fully discrete approximation with respect to the continuous-in-time one, for which there are already estimates in the literature.}

Finally, following \cite{novo_rubino}, we analyze the case in which stabilized approximations are computed both for the FOM
and the POD-ROM. More precisely, the considered finite element method is based on a Galerkin discretization plus 
grad-div stabilization with pairs of inf-sup stable elements,   {and, as in references \cite{novo_rubino,pod_da_nos},} for the POD-ROM we also use grad-div stabilization. In 
this way, the constants in the error bounds for the snapshots do not depend explicitly on inverse powers of the viscosity, i.e., they do not blow up for small viscosity coefficients, see \cite{NS_grad_div}.
Adapting the results from \cite{novo_rubino}, the same holds for the error bounds of the POD-ROM. 
The importance of such so-called robust methods is stated in the survey \cite{GJN21}: \textit{In the case of small
viscosity coefficients and coarse grids, only robust estimates provide useful information about the behavior of a
numerical method on coarse grids if the analytic solution is smooth.} {In reference \cite{pod_da_nos} a POD scheme with data assimilation was analyzed. In the method of \cite{pod_da_nos}, grad-div stabilization was added to the POD method, but the plain Galerkin
method was used for the FOM so that the final error bounds depend on inverse powers of the viscosity.} {As opposed to the present paper, neither in reference \cite{novo_rubino} nor in \cite{pod_da_nos}, snapshots approaching the time derivative are included and no pointwise estimates are proved}.

In the numerical studies, we compare the different approaches obtained by taking $X=H_0^1(\Omega)^d$ and $X=L^2(\Omega)^d$ in combination with one of the following sets:
the set of snapshots at different times, the set of difference quotients, and the set 
of Galerkin time derivatives. We cannot deduce from these studies that any of the approaches
is much better than the other ones and the necessary comprehensive numerical studies are outside the 
scope of this paper, which is a rigorous numerical
analysis from which interesting properties and sharp bounds for the different methods can be deduced. 

The paper is organized as follows. In Section \ref{sec:PN} we state some preliminaries and notations. The POD method and
some a priori bounds for the projection of the FOM approximation onto the POD space are shown in Section \ref{sec:POD}. The analysis
of the POD method is included in Section \ref{sec:pod_rom} with a first subsection for the case $X=H_0^1(\Omega)^d$ and a second
one for the case $X=L^2(\Omega)^d$. As stated above, Section \ref{sec:num} is devoted to study the performance of the methods
with some numerical experiments. Finally, 
we have included an appendix in which we get robust bounds for the second time derivative of the FOM approximation in some norms,
since this time derivative appears in our bounds in the a priori error analysis.

\section{Preliminaries and notations}\label{sec:PN}

Standard symbols will be used for Lebesgue and Sobolev spaces, with the usual convention that 
$W^{s,2}(\Omega)= H^s(\Omega)$, $s\ge 1$. The inner product in $L^2(\Omega)^d$, $d\ge 1$, is denoted 
by $(\cdot,\cdot)$, {and the corresponding norm by $\|\cdot\|_0$}.

The following Sobolev imbeddings \cite{Adams} will be used in the analysis: For
$q \in [1, \infty)$, there exists a constant $C=C(\Omega, q)$ such
that
\begin{equation}\label{sob1}
\|v\|_{L^{q'}} \le C \| v\|_{W^{s,q}}, \,\,\quad
\frac{1}{q'}
\ge \frac{1}{q}-\frac{s}{d}>0,\quad q<\infty, \quad v \in
W^{s,q}(\Omega)^{d}.
\end{equation}
We will denote by $C_p$ the constant in the Poincar\'e inequality
\begin{equation}\label{poincare}
\|\bv\|_0\le C_p\|\nabla \bv\|_0,\quad \bv\in H_0^1(\Omega)^d.
\end{equation}
The following inequality can be found in  \cite[Remark~3.35]{John}
\begin{equation}\label{diver_vol}
\|\nabla \cdot \bv\|_0\le \|\nabla   \bv \|_0,\quad \bv\in H_0^1(\Omega)^d.
\end{equation}
Let us denote by ${\boldsymbol V} = H_0^1(\Omega)^d$ and 
$Q=L_0^2(\Omega)=\{q\in L^2(\Omega)\ \mid \ (q,1)=0\}$.

Let $\mathcal{T}_{h}=(\sigma_j^h,\phi_{j}^{h})_{j \in J_{h}}$, $h>0$, be a family of partitions of 
$\overline\Omega$, where $h$ denotes the maximum diameter of the mesh cells $\sigma_j^h\in \mathcal{T}_{h}$, 
and $\phi_j^h$ are the mappings from the reference simplex $\sigma_0$ onto $\sigma_j^h$.
We shall assume that the family of partitions is shape-regular and  quasi-uniform. 
On these partitions, we define the following finite element spaces
\begin{eqnarray}
Y_h^l&=& \left\{v_h\in C^0(\overline\Omega)\ \mid \ {v_h}_{\mid_K}\in {\Bbb P}_l(K),\ \forall\ K\in \mathcal T_h\right\}, \ l\ge 1,\quad {\boldsymbol Y}_h^l=\left(Y_h^l\right)^d,\nonumber\\
{\boldsymbol X}_h^l&=&{\boldsymbol Y}_h^l\cap H_0^1(\Omega)^d,  \quad
Q_h^l = Y_h^l\cap L_0^2(\Omega), \nonumber\\
{{\boldsymbol V}_h^l}&=&{\boldsymbol X}_h^l\cap \left\{ {\boldsymbol v}_{h} \in H_0^1(\Omega)^d \ \mid \
(q_{h}, \nabla\cdot{\boldsymbol v}_{h}) =0  \ \forall\ q_{h} \in Q_{h}^{l-1}
\right\},\quad l\geq 2.\label{eq:V}
\end{eqnarray}
The space ${\boldsymbol V}_h^l$ is the space of discretely divergence-free functions. 

Since the family of partitions is quasi-uniform, the following inverse
inequality holds for each {$\bv_{h} \in {\boldsymbol Y}_{h}^{l}$}, e.g., see \cite[Theorem~3.2.6]{Cia78},
\begin{equation}
\label{inv} | \bv_{h} |_{W^{m,p}(K)} \leq c_{\mathrm{inv}}
h_K^{n-m-d\left(\frac{1}{q}-\frac{1}{p}\right)}
|\bv_{h}|_{W^{n,q}(K)},
\end{equation}
where $0\leq n \leq m \leq 1$, $1\leq q \leq p \leq \infty$, and $h_K$
is the diameter of~$K \in \mathcal T_h$.

The analysis uses a modified Stokes projection  $\bs_h^m\ :\ {\boldsymbol V}\rightarrow {\boldsymbol V}_{h,l}$ that was introduced in \cite{grad-div1} and 
that is defined by 
\begin{equation}\label{stokespro_mod_def}
(\nabla \bs_h^m,\nabla \bvar_h)=(\nabla \bu,\nabla \bvar_h),\quad \forall\
\bvar_{h} \in {\bV_h^l}.
\end{equation}
This projection satisfies the following error bound, see \cite{grad-div1},
\begin{equation}
\|\bu-\bs_h^m\|_0+h\|\bu-\bs_h^m\|_1\le C\|\bu\|_j h^j,\quad
1\le j\le l+1.
\label{stokespro_mod}
\end{equation}
From \cite{chenSiam}, we also have
\begin{equation}
\|\nabla \bs_h^m\|_\infty\le C\|\nabla \bu\|_\infty \label{cotainfty1},
\end{equation}
and from \cite[Lemma~3.8]{bosco_titi_yo}
\begin{eqnarray}
\label{cota_sh_inf_mu}
\|\bs_h^m\|_\infty  & \le & C(\|\bu\|_{d-2}\|\bu\|_2)^{1/2},
\\
\label{la_cota_mu}
\|\nabla\bs_h^m\|_{L^{2d/(d-1)}} & \le &
 C\bigl(\|\bu\|_1\|\bu\|_2\bigr)^{1/2},
\end{eqnarray}
where all constants~$C$ in \eqref{cotainfty1} -- \eqref{la_cota_mu} are independent of~$\nu$.
{In the sequel, for simplicity, we will denote by $C(\bu,p)$ a generic constant depending
on some norms of the true velocity and pressure.}

We consider the  mixed finite element pair known as Hood--Taylor elements \cite{BF,hood0} $({\boldsymbol X}_h^l, Q_{h}^{l-1})$, $l \ge 2$.
For these elements a uniform inf-sup condition is satisfied (see \cite{BF}), that is, there exists a constant $\beta_{\rm is}>0$ independent of the mesh size $h$ such that
\begin{equation}\label{lbbh}
 \inf_{q_{h}\in Q_{h}^{l-1}}\sup_{\bv_{h}\in{\boldsymbol X}_h^l}
\frac{(q_{h},\nabla \cdot \bv_{h})}{{\|\nabla\bv_{h}\|_{0}}
\|q_{h}\|_0} \geq \beta_{\rm{is}}.
\end{equation}

As a direct method, or full order method,  
we consider a Galerkin method with grad-div stabilization. The semi-discrete method reads as follows:
Find $(\bu_h,p_h)\in {\boldsymbol X}_h^l\times Q_h^{l-1}$ such that 
\begin{eqnarray}\label{eq:gal_grad_div}
\left(\bu_{h,t},\bv_h\right)+\nu(\nabla \bu_h,\nabla \bv_h)+b_h(\bu_h,\bu_h,\bv_h)
\\
-(p_h,\nabla \cdot \bv_h)+
\mu(\nabla \cdot\bu_h,\nabla \cdot \bv_h) & = & ({\boldsymbol f},\bv_h) \quad \forall\ \bv_h\in {\boldsymbol X}_h^l,\nonumber\\
(\nabla \cdot \bu_h,q_h)&=&0 \quad \forall\ q_h\in Q_h^{l-1},\nonumber
\end{eqnarray}
where $\mu$ is the positive grad-div stabilization parameter { and
\[
b_h(\bu_h,\bu_h,\bv_h)=((\bu_h\cdot\nabla) \bu_h,\bv_h)+\frac{1}{2}((\nabla \cdot \bu_h)\bu_h,\bv_h).
\]
}
It is well-known that considering the discretely divergence-free space $\bV_h^l$, we can remove the pressure from \eqref{eq:gal_grad_div} since
$\bu_h\in \bV_h^l$ satisfies
\begin{equation}\label{eq:gal_grad_div2}
\left(\bu_{h,t},\bv_h\right)+\nu(\nabla \bu_h,\nabla \bv_h)+b_h(\bu_h,\bu_h,\bv_h)+
\mu(\nabla \cdot\bu_h,\nabla \cdot \bv_h)=({\boldsymbol f},\bv_h), \quad \forall\ \bv_h\in {\bV}_h^l.
\end{equation}
For this method the following bound holds, see \cite{NS_grad_div},
\begin{equation}\label{eq:cota_grad_div}
\|\bu(\cdot,t)-\bu_h(\cdot,t)\|_0+ h\|\bu(\cdot,t)-\bu_h(\cdot,t)\|_1\le C(\bu,p,l+1) h^{l},\quad t\in (0,T],
\end{equation}
where the constant $C(\bu,p,l+1)$ does not explicitly depend on inverse powers of $\nu$. {As
detailed in \cite{NS_grad_div} $C(\bu,p,l+1)$ depends both on $\mu$ and $\mu^{-1}$.  However, assuming as in \cite{NS_grad_div}
the grad-div parameter is independent of the mesh size, we can omit the dependence of the constant on $\mu$}.
Actually, only the first term on the left-hand side of \eqref{eq:cota_grad_div} is considered in 
\cite{NS_grad_div} 
but the estimate for the second term follows then from \eqref{stokespro_mod}
and the inverse inequality \eqref{inv}. Numerical studies presented in 
\cite{GJN21} show that the estimate from \cite{NS_grad_div} is sharp.
From the error analysis performed in \cite{NS_grad_div}, it can be seen that an optimal order for the 
error of the velocity gradient, in $L^2(0,T;L^2(\Omega)^d)$, 
is obtained with a constant that 
depends on $\nu^{-1}$, i.e., this estimate is not robust. 

{Finally we will use the standard notation $H^s(0,T;X)$ (see e.g. \cite[Section 5.9.2]{Evans}) for
$X$ being a Banach space.}

\begin{remark}\label{comp_ut}\rm {We briefly comment on the computation of the time derivatives from the Galerkin equations. 
Equipping ${\boldsymbol X}_h^l$ and $Q_h^l$ with a standard basis, the solution of 
\eqref{eq:gal_grad_div} can be represented with coefficient vectors $\underline{\bu}_h(t)$
and $\underline{p}_h(t)$, respectively. In terms of these vectors, equations~\eqref{eq:gal_grad_div} become
\begin{equation}
\label{system_nodes}
M_h\frac{\displaystyle \vphantom{\big|_{1}}d}{\displaystyle \vphantom{\big|_{1}}dt}\underline \bu_h -  B_h^T \underline p_h =F_h(\underline \bu_h),\quad
B_h\underline\bu_h =0,
\end{equation}
where $M_h$ is the mass matrix formed by the velocity
basis functions and $B_h$ is the discrete divergence matrix, 
(see e.g., \cite[\S~7.1]{John}). We notice that~\eqref{system_nodes} is a system of differential-algebraic equations (DAE) where approximation to its solution (once a compatible initial condition is specified) can be obtained at certain time levels (see e.g., \cite{BDF2} for details). Once the vectors of nodal values~$\underline\bu_h$ and~$\underline p_h$ are obtained for a given time level, the corresponding vector of nodal values of the time derivative $d\bu_h/dt$ can be obtained by finding the solution~$\underline x$ 
of the linear system $M_h \underline{x} = F_h(\underline \bu_h) + B_h^T\underline p_h$.}
\hspace*{\fill}$\Box$\end{remark}
\section{Proper orthogonal decomposition}\label{sec:POD}

We consider a POD method.
Let us fix $T>0$ and $M>0$ and take $\Delta t=T/M$. For $N=M+1$ we define the following space
\[
{\cal \bU}=\mbox{span}\left\{\by_h^1,\by_h^2,\ldots,\by_h^N\right\},
\]
with
\[
\by_h^1=\sqrt{N}\bu_h^0,\quad \by_h^j=\tau\bu_{h,t}^{j+1},\ j=2,\ldots,N,
\]
so that
{
\[
{\cal \bU}= 
\mbox{span}\left\{\sqrt{N}\bu_h^0,\tau\bu_{h,t}^1,\ldots,\tau\bu_{h,t}^M\right\},
\]
}
where we use the notation $\bu_h^j=\bu_h(\cdot,t_j)$ for the approximations at time instance $t_j=j\Delta t$ and $\bu_{h,t}^j=\bu_{h,t}(\cdot,t_j)$ are the 
snapshots of the temporal derivatives.  The factor $\tau$ in front of the temporal derivatives is a time scale and it makes the 
snapshots dimensionally correct, 
{i.e., all members of the set that defines $\boldsymbol U$ are of the 
same physical quantity (here, velocities). The factor~$\tau$ may also be used to alter the size of
the time derivatives relative to the initial velocity, which has an impact on the eigenvalues of the correlation matrix and on the POD basis.}
 {We have not explored this possibility in the numerical experiments of Section 5 and have chosen $\tau=T$ for simplicity}.
{In the sequel,} $d_v$ {denotes} the dimension of $\cal \bU$.

The correlation  matrix corresponding to the snapshots is given by ${K^v}=((k_{i,j}^v))\in {\mathbb R}^{N\times N}$,
with the entries 
\[
k_{i,j}^v=\frac{1}{N}\left(\by_h^i,\by_h^j\right)_X, \quad i,j=1,\ldots,N, 
\]
and $(\cdot,\cdot)_X$ is the inner product in $X$, which is either $L^2(\Omega)^d$ or $H_0^1(\Omega)^d$.
Following \cite{kunisch}, we denote by
$ \lambda_1\ge  \lambda_2,\ldots\ge \lambda_{d_v}>0$ the positive eigenvalues of \textcolor{red}{$K^v$} and by
$\bv_1,\ldots,\bv_{d_v}\in {\mathbb R}^{N}$  associated eigenvectors of Euclidean norm $1$.  
Then, the (orthonormal) POD basis functions of $\cal \bU$ are given by
\begin{equation}\label{lachi}
\bvar_k=\frac{1}{\sqrt{N}}\frac{1}{\sqrt{\lambda_k}}\sum_{j=1}^{N} v_k^j \by_h^j,
\end{equation}
where $v_k^j$ is the $j$-th component of the eigenvector $\bv_k$.
The following error estimate is known from \cite[Proposition~1]{kunisch}
\begin{eqnarray}\label{cota_ku}
\frac{1}{N}\sum_{j=1}^N\left\|\by_h^j-\sum_{k=1}^r(\by_h^j,\bvar_k)_X\bvar_k\right\|_{X}^2=\sum_{k=r+1}^{d_v}\lambda_k,
\end{eqnarray}
from which one can deduce
\begin{equation}
\left\|\bu_h^0-\sum_{k=1}^r(\bu_h^0,\bvar_k)_X\bvar_k\right\|_{X}^2+\frac{\tau^2}{M+1}\sum_{j=1}^M\left\|\bu_{h,t}^{j}-\sum_{k=1}^r(\bu_{h,t}^{j},\bvar_k)_X\bvar_k\right\|_{X}^2=\sum_{k=r+1}^{d_v}\lambda_k.\label{eq:cota_pod_deriv}
\end{equation}
{In the sequel, we will denote by
$ {\cal \bU}^r= \mbox{span}\{\bvar_1,\bvar_2,\ldots,\bvar_r\}$, $1\le r\le d_v,
$
and by $P_r^v\ : \  {\boldsymbol X}_h^l  \to {\cal \bU}^r$,  the $X$-orthogonal projection onto ${\cal \bU}^r$. 
Then \eqref{cota_ku} can be written as
\[
\frac{1}{N}\sum_{j=1}^N\left\|\by_h^j-P_r^v \by_h^j\right\|_{X}^2=\sum_{k=r+1}^{d_v}\lambda_k.
\]
A generalization of the above equality can be found in \cite[Lemma 2.2]{koc_rubino_et_al}. Let $W$ be a Hilbert space
with ${\cal \bU}\subset W$ and $R_r^v: W\rightarrow W$ be a bounded linear 
 projection onto ${\cal \bU}^r$ then
\begin{eqnarray}\label{bound_point}
\frac{1}{N}\sum_{j=1}^N\left\|\by_h^j-R_r^v \by_h^j\right\|_{W}^2=\sum_{k=r+1}^{d_v}\lambda_k\|\bvar_k-R_r^v\bvar_k\|_{W}^2.
\end{eqnarray}
We observe that \eqref{cota_ku} is \eqref{bound_point} in the case $R_r^v=P_r^v$ and $W=X$.
}

We will denote the mass matrix of the POD basis by  $M^v=((m_{i,j}^v))\in {\mathbb R}^{d_v\times d_v}$, where 
$m_{i,j}^v=( \bvar_j,\bvar_i)_X$. 
In the case $X=H_0^1(\Omega)^d$, for any $\bv \in {\cal \bU}$,
the following inverse inequality holds, see \cite[Lemma~2]{kunisch},
\begin{equation}\label{eq:inv_M}
\|\nabla \bv \|_0\le \sqrt{\|(M^v)^{-1}\|_2}\|\bv\|_0.
\end{equation}

The stiffness matrix of the POD basis is given by $S^v=((s_{i,j}^v))\in {\mathbb R}^{d_v\times d_v}$, with the entries
$s_{i,j}^v=(\nabla \bvar_j,\nabla \bvar_i)_X$.
If $X=L^2(\Omega)^d$, the following inequality holds for all $\bv \in {\cal \bU}$, 
see \cite[Lemma 2]{kunisch},
\begin{equation}\label{eq:inv_S}
\|\nabla \bv \|_0\le \sqrt{\|S^v\|_2}\|\bv\|_0.
\end{equation}

The following lemma will be the basis for proving 
pointwise in time estimates. We follow \cite[Lemma~3.3]{locke_singler} (see also \cite[Lemma 3.6]{koc_rubino_et_al}).

\begin{lemma}\label{le:our_etal}
Let $T>0$, $\Delta t=T/M$, $t^n=n\Delta t$, $n=0,1,\ldots M$, let $X$ be a {Banach} space, $\bz=\bz(t,\bx)\in H^2(0,T;X)$. Then, the following estimate holds
\begin{equation}\label{eq:zetast}
\max_{0\le k\le {M} }\|\bz^k\|_X^2 \le  {3}\|\bz^0\|_X^2+\frac{3 T^2}{M}\sum_{n=1}^M \| \bz_t^n\|_X^2
+\frac{4T}{3}(\Delta t)^2\int_0^T\|\bz_{tt}(s)\|_X^2\ ds,
\end{equation}
{where $\bz^n=\bz(t_n,\cdot)$, $\bz_t^n=\bz_t(t_n,\cdot)$}.
\end{lemma}

\begin{proof}
 For each $k$ we have
\begin{equation}\label{eq:fun}
\bz^k=\bz^0+\int_{t_0}^{t_k}\bz_t\ ds.
\end{equation}
Adding and subtracting terms leads to 
\begin{equation}\label{zeta_uno}
\bz^k=\bz^0+\int_{t_0}^{t_k}\bz_t\ ds=\bz^0+\sum_{{n=1}}^{k}\Delta t \bz_t^n+ \sum_{{n=1}}^{k} \int_{{t_{n-1}}}^{{t_{n}}}\left(\bz_t(s)-\bz_t(t_n)\right)\ ds.
\end{equation}
To bound the last term on the right-hand side, we first notice that for $s\in[t_{n-1},t_n]$
\[
\bz_t(t_n)-\bz_t(s) = \int_{s}^{t_n} \bz_{tt}(\sigma)\ d\sigma. 
\]
With the Cauchy--Schwarz inequality, it follows that 
\begin{eqnarray*}
\left\|\int_{{t_{n-1}}}^{{t_{n}}}
(\bz_t(s)-\bz_t(t_n))\ ds\right\|_X &\le&
 \int_{{t_{n-1}}}^{{t_{n}}}\left\|\bz_t(t_n)-\bz_t(s)\right\|_X\ ds
 \\
&\le& \int_{{t_{n-1}}}^{{t_{n}}}{(t_{n}-s)}^{1/2}  \left(\int_{{s}}^{{t_n}}\left\|\bz_{tt}(\sigma)\right\|_X^2\ d\sigma\right)^{1/2}
\ ds\\
&\le& \frac{2}{3}(\Delta t)^{3/2} \left(\int_{t_{{n-1}}}^{t_{{n}}}\left\|\bz_{tt}(s)\right\|_X^2\ ds\right)^{1/2}.
\end{eqnarray*}
Consequently, for the last term on the right-hand side of \eqref{zeta_uno}, we obtain, using 
the Cauchy--Schwarz inequality for sums, 
\begin{eqnarray*}
\left\| 
{\sum_{n=1}^{k} }\left( \int_{t_{n-1}}^{t_{n}}
(\bz_t(s)-\bz_t(t_n))\ ds\right) \right\|_X
&\le& {\sum_{n=1}^{k}} \left\|\int_{t_{n-1}}^{t_{n}}
(\bz_t(s)-\bz_t(t_n))\ ds\right\|_X
\\
&\le &
 \frac{2}{3}{\sum_{n=1}^{k}}\left[ (\Delta t)^{3/2} \left(\int_{t_{n-1}}^{t_{n}}\left\|\bz_{tt}(s)\right\|_X^2\  ds\right)^{1/2}\right]
 \\
 &\le &
 \frac{2}{3}\left({\sum_{n=1}^{k}}(\Delta t)^{3} \right)^{1/2}\left(\int_{t_0}^{t_{k}}\left\|\bz_{tt}(t)\right\|_X^2\ ds\right)^{1/2}
 \\
 &\le &
  \frac{2}{3}T^{1/2}\Delta t \left(\int_{t_0}^{t_{k}}\left\|\bz_{tt}(t)\right\|_X^2\ ds\right)^{1/2}.
\end{eqnarray*}
Taking norms in \eqref{zeta_uno} gives the estimate
\begin{eqnarray}\label{drei}
\|\bz^k\|_X &\le &\|\bz^0\|_X+{\sum_{n=1}^{k}}\Delta t \|\bz_t^n\|_X+\frac{2}{3}T^{1/2}\Delta t \left(\int_{t_0}^{t_{k}}\left\|\bz_{tt}(t)\right\|_X^2\ ds\right)^{1/2}\nonumber\\
&\le&\|\bz^0\|_X+T^{1/2}(\Delta t)^{1/2}\left(\sum_{n=1}^M \|\bz_t^n\|_X^2\right)^{1/2}+\frac{2}{3}T^{1/2}\Delta t \left(\int_{0}^{T}\left\|\bz_{tt}(t)\right\|_X^2\ ds\right)^{1/2},
\end{eqnarray} 
from which we conclude \eqref{eq:zetast}.
\end{proof}

The above lemma also holds true changing the initial value by the mean value.
\begin{lemma}\label{le:our_etal_mean}
With the assumptions of Lemma~\ref{le:our_etal}, it holds 
\begin{equation}\label{eq:zetast_mean}
\max_{0\le k\le N }\|\bz^k\|_X^2 \le  {3}\|\overline\bz\|_X^2+\frac{12 T^2}{M}\sum_{n=1}^M \| \bz_t^n\|_X^2+\frac{16T}{3}(\Delta t)^2\int_0^T\|\bz_{tt}(s)\|_X^2\ ds,
\end{equation}
where
$
\overline \bz=\frac{1}{M+1}\sum_{j=0}^M\bz^j$.
\end{lemma}

\begin{proof}
We first observe that \eqref{eq:fun}  gives
\[
\overline \bz=\frac{1}{M+1}\sum_{j=0}^M\left(\bz^0+\int_{t_0}^{t_j}\bz_t\ ds\right).
\]
Then
\[
\overline \bz=\bz^0+\frac{1}{M+1}\left\{\int_{t_0}^{t_1}\bz_t \ ds+\int_{t_0}^{t_2}\bz_t \ ds+\ldots+\int_{t_0}^{t_M}\bz_t \ ds\right\},
\]
so that 
\[
\|\bz^0\|_X\le \|\overline \bz\|_X+\frac{1}{M+1}\sum_{j=1}^M\left\|\int_{t_0}^{t_j}\bz_t \ ds\right\|_X.
\]
Since, {in the argument from \eqref{zeta_uno} to \eqref{drei} in the proof of  Lemma \ref{le:our_etal},
we have obtained }
\begin{equation}\label{cota_int}
\left\|\int_{t_0}^{t_j}\bz_t \right\|_X\le T^{1/2}(\Delta t)^{1/2}\left(\sum_{n=1}^M \|\bz_t^n\|_X^2\right)^{1/2}+\frac{2}{3}T^{1/2}\Delta t \left(\int_{0}^{T}\left\|\bz_{tt}(t)\right\|_X^2\ ds\right)^{1/2},
\end{equation}
it follows that 
\begin{equation}\label{eq:also}
\|\bz^0\|_X \le  \|\overline\bz\|_X+T^{1/2}(\Delta t)^{1/2}\left(\sum_{n=1}^M \|\bz_t^n\|_X^2\right)^{1/2}
+\frac{2}{3}T^{1/2}\Delta t \left(\int_{0}^{T}\left\|\bz_{tt}(t)\right\|_X^2\ ds\right)^{1/2}.
\end{equation}
Now, taking into account \eqref{eq:fun},  \eqref{cota_int}, and \eqref{eq:also}, we obtain for any $k$ 
\[
\|\bz^k\|_X \le \|\overline\bz\|_X+2T^{1/2}(\Delta t)^{1/2}\left(\sum_{n=1}^M \|\bz_t^n\|_X^2\right)^{1/2}
+\frac{4}{3}T^{1/2}\Delta t \left(\int_{0}^{T}\left\|\bz_{tt}(t)\right\|_X^2\ ds\right)^{1/2},
\]
from which we conclude \eqref{eq:zetast_mean}.
\end{proof}

\begin{remark}\rm
In the sequel we will apply Lemma~\ref{le:our_etal} and assume that we have in the set of snapshots
$\by_h^1=\sqrt{N}\bu_h^0$.
However, applying Lemma~\ref{le:our_etal_mean} instead of Lemma~\ref{le:our_etal}, we can substitute the first snapshot by $\by_h^1=\sqrt{N}\overline{\bu}_h$, where $\overline{\bu}_h$ is the mean value 
$
\overline{\bu}_h=\frac{1}{M+1}\sum_{j=0}^M \bu_h^j$,
to obtain the same results. 
The set of snapshots in this case is
\[
{\cal \bU}=\mbox{span}\left\{\by_h^1,\by_h^2,\ldots,\by_h^N\right\} = 
\mbox{span}\left\{\sqrt{N}\overline\bu_h,\tau\bu_{h,t}^1,\ldots,\tau\bu_{h,t}^M\right\}.
\]
It is standard in numerical simulations to subtract the mean and apply the POD method to the
fluctuations, see \cite{John_et_al_vp}, which by definition have zero mean. Then, the new approach we propose,
in which the first snapshot is a weighted mean and the rest are weighted time derivatives, has the advantage that applying the POD method to
the fluctuations only includes snapshots for the approximations to the time derivatives (since the mean is zero). This procedure was applied in
our numerical studies.
\hspace*{\fill}$\Box$\end{remark}

{
\begin{lemma} \label{le:poin} The following bounds hold
\begin{equation}\label{eq:bound_2nd_term}
\max_{0\le n\le M }\|P_r^v\bu_h^n-\bu_h^n\|_X^2
\le C_{X}^2:=\left(3+6\frac{T^2}{\tau^2}\right)\sum_{k={r+1}}^{d_v}\lambda_k
+\frac{16T}{3}(\Delta t)^2 \int_0^T\|\bu_{h,tt}(s)\|_X^2\ ds
\end{equation}
and
\begin{equation}
\label{eq:cota_pod_0}
\frac{1}{M}\sum_{j=1}^M\left\|\bu_{h}^{j}-\sum_{k=1}^r(\bu_{h}^{j},\bvar_k)_X\bvar_k\right\|_{X}^2\le C_{X}^2.
\end{equation}
\end{lemma}
}
\begin{proof}
Taking $\bz=P_r^v\bu_h-\bu_h$ in \eqref{eq:zetast} and applying \eqref{eq:cota_pod_deriv} ({taking into account that
$(M+1)/M\le2$}) yields
\begin{eqnarray*}
\max_{0\le n\le M }\|P_r^v\bu_h^n-\bu_h^n\|_X^2&\le& \left(3+6\frac{T^2}{\tau^2}\right)\sum_{k={r+1}}^{d_v}\lambda_k\nonumber\\
&&+\frac{4T}{3}(\Delta t)^2\int_0^T\|P_r^v\bu_{h,tt}(s)-\bu_{h,tt}(s)\|_X^2\ ds.
\end{eqnarray*}
Let us bound the second term on the right-hand side above. 
Using the triangle inequality and taking into account that
$\|P_r^v\bw\|_X^2\le \|\bw\|_X^2$, we conclude that 
\[
\|P_r^v\bu_{h,tt}(s)-\bu_{h,tt}(s)\|_X^2 \le  2\|P_r^v\bu_{h,tt}(s)\|_X^2+2 \|\bu_{h,tt}(s)\|_X^2
\le 4 \|\bu_{h,tt}(s)\|_X^2,
\]
Finally, from \eqref{eq:bound_2nd_term} we obtain \eqref{eq:cota_pod_0}.
\end{proof}
The integral term in \eqref{eq:bound_2nd_term} and \eqref{eq:cota_pod_0} is bounded in the appendix. In case one uses difference
quotients instead of Galerkin time derivatives in the set of snapshots this term does not appear.

\subsection{A priori bounds for the projections $P_r^v \bu_h^j$}

This subsection
is devoted to proving a priori bounds for the orthogonal projections $P_r^v\bu_h^j$, $j=0,\ldots,M$.
These bounds are
obtained from a priori bounds for the Galerkin approximation $\bu_h^j$, $j=0,\cdots,M$. We argue as in \cite{novo_rubino}. 
Then, we start getting a priori
bounds for the stabilized approximation $\bu_h^n$. We follow the same arguments we introduced in \cite{pod_da_nos}
{with the difference that since we now compare with the continuous-in-time approximation $\bu_h^j$ the bounds
are independent of the time step discretization}.

We start with the $L^\infty$ norm. Using \eqref{inv}, \eqref{cota_sh_inf_mu}, \eqref{eq:cota_grad_div} and  \eqref{stokespro_mod},
we get
\begin{eqnarray}\label{eq:uh_infty}
\max_{0\le j\le M}\|\bu_h^j\|_\infty&\le& \|\bu_h^j-\bs_h^m(\cdot,t_j)\|_\infty+\|\bs_h^m(\cdot,t_j)\|_\infty
\nonumber\\
&\le& C h^{-d/2}\|\bu_h^j-\bs_h^m(\cdot,t_j)\|_0+C(\|\bu^j\|_{d-2}\|\bu^j\|_2)^{1/2}
\nonumber\\
&\le& C h^{-d/2}C(\bu,p,3)h^{2}+C(\|\bu^j\|_{d-2}\|\bu^j\|_2)^{1/2}\\
&\le& C_{\bu,{\rm inf}}:=C\left( C(\bu,p,3)+\left(\|\bu\|_{L^\infty(H^{d-2})}\|\bu\|_{L^\infty(H^{2})}\right)^{1/2}\right),\nonumber
\end{eqnarray}
{where in the third line we have bounded $\|\bu_h^j-\bs_h^m(\cdot,t_j)\|_0\le
\|\bu_h^j-\bu(\cdot,t_j)\|_0+\|\bu(\cdot,t_j)-\bs_h^m(\cdot,t_j)\|_0$.}
The $L^\infty$ norm of the gradient is bounded in a similar way. 
Using \eqref{inv}, \eqref{cotainfty1}, \eqref{eq:cota_grad_div}, and  \eqref{stokespro_mod},
we obtain 
\begin{eqnarray}\label{eq:grad_uh_infty}
\max_{0\le j\le M}\|\nabla\bu_h^j\|_\infty&\le& \|\nabla\bu_h^j-\nabla\bs_h^m(\cdot,t_j)\|_\infty+\|\nabla\bs_h^m(\cdot,t_j)\|_\infty
\nonumber\\
&\le& C h^{-d/2}\|\bu_h^j-\bs_h^m(\cdot,t_j)\|_1+C\|\nabla \bu^j\|_\infty\nonumber\\
&\le& C h^{-d/2}C(\bu,p,d+1)h^{d-1}+C\|\nabla \bu^j\|_\infty\nonumber\\
&\le& C_{\bu,1,{\rm inf}}:= C\left(C(\bu,p,d+1)+\|\nabla  \bu\|_{L^\infty(L^\infty)}\right),
\end{eqnarray}
{where the same argument as before has been applied to bound $\|\bu_h^j-\bs_h^m(\cdot,t_j)\|_1$}.
Note that the estimate for $d=3$ requires the use of cubic elements for the velocity. 
Finally, applying \eqref{inv}, \eqref{la_cota_mu}, \eqref{eq:cota_grad_div}, and  \eqref{stokespro_mod}
leads to the following bound of the $L^{2d/(d-1)}$ norm of the velocity gradient
\begin{eqnarray}\label{eq:uh_2d_dmenosuno}
\max_{0\le j\le M}\|\nabla \bu_h^j\|_{L^{2d/(d-1)}}&\le& \|\nabla(\bu_h^j-\bs_h^m(\cdot,t_j))\|_{L^{2d/(d-1)}}+\|\nabla \bs_h^m(\cdot,t_j)\|_{L^{2d/(d-1)}}
\nonumber\\
&\le& C h^{-1/2}\|\bu_h^j-\bs_h^m(\cdot,t_j)\|_1+C\bigl(\|\bu\|_1\|\bu\|_2\bigr)^{1/2}\nonumber\\
&\le& C h^{-1/2}C(\bu,p,3)h+C\bigl(\|\bu^j\|_1\|\bu^j\|_2\bigr)^{1/2}\\
&\le& C_{\bu,{\rm ld}}:=C\left( C(\bu,p,3)+\left(\|\bu\|_{L^\infty(H^{1})}\|\bu\|_{L^\infty(H^{2})}\right)^{1/2}\right).\nonumber
\end{eqnarray}

Now, we prove a priori bounds in the same norms for
\[
P_r^v \bu_h^j=(P_r^v \bu_h^j-\bu_h^j)+\bu_h^j.
\]
Since we have already proved error bounds for the second term on the right-hand side, we only need to bound the first one. 
First, we consider the case
$X=L^2(\Omega)^d$. From \eqref{eq:bound_2nd_term} {in Lemma \ref{le:poin}} we get for $j=0,\ldots,M$,
\begin{eqnarray}\label{cota_bos1}
\|\bu_h^j-P_r^v \bu_h^j\|_0\le C_{L^2}.
\end{eqnarray}
Applying the inverse inequality \eqref{inv} gives
\[
\|P_r^v \bu_h^j\|_\infty\le \|\bu_h^j\|_\infty+c_{\rm inv} h^{-d/2}\|\bu_h^j-P_r^v \bu_h^j\|_0.
\]
{Utilizing \eqref{eq:uh_infty} and \eqref{cota_bos1} yields} 
\begin{eqnarray}\label{eq:cotaPlinf2} C_{\rm inf}:=\max_{0\le j\le M} \|P_r^v \bu_h^j\|_\infty\le
C_{\bu,{\rm inf}}+c_{\rm inv} h^{-d/2} C_{L^2},
\end{eqnarray}
so that 
\begin{equation}
\label{eq:cotaPlinf2_sum}
K_{\rm inf}:=\Delta t\sum_{j=0}^{M}\|P_r^v \bu_h^j\|_\infty^2 \le  TC_{\rm inf}^2.
\end{equation}
Now, we observe that from  \eqref{eq:inv_S} and \eqref{cota_bos1}, we get
\begin{equation}\label{ref}
\|\nabla(\bu_h^j-P_r^v \bu_h^j)\|_0\le\|S^v\|_2^{1/2}C_{L^2}.
\end{equation}
{Let us observe that although $\|S^v\|_2^{1/2}$ increases with $r$, the constant $C_{L^2}$ in
\eqref{eq:bound_2nd_term} decreases with $r$. Alternatively, we can argue as in Lemma~\ref{le:poin} but applying 
Lemma~\ref{le:our_etal} with $X=H_0^1(\Omega)^d$ together with \eqref{bound_point} with $R_r^v=P_r^v$ and $W=H_0^1(\Omega)^d$
to get
\begin{eqnarray*}
\max_{0\le n\le M }\|\nabla(P_r^v\bu_h^n-\bu_h^n)\|_0^2&\le& \left(3+6\frac{T^2}{\tau^2}\right)\sum_{k={r+1}}^{d_v}\lambda_k\|\nabla \bvar_k\|_0\nonumber\\
&&+\frac{4T}{3}(\Delta t)^2\int_0^T\|\nabla(P_r^v\bu_{h,tt}(s)-\bu_{h,tt}(s))\|_0^2\ ds.
\end{eqnarray*}
However, after having used
\[
\|\nabla(P_r^v\bu_{h,tt}(s)-\bu_{h,tt}(s))\|_0^2 \le  2\|\nabla P_r^v\bu_{h,tt}(s)\|_0^2+2 \|\nabla\bu_{h,tt}(s)\|_0^2,
\]
we apply \eqref{eq:inv_S} to bound the first term as follows
\[
\|\nabla P_r^v\bu_{h,tt}(s)\|_0^2\le \|S^v\|_2\| P_r^v\bu_{h,tt}(s)\|_0^2
\le \|S^v\|_2\|\bu_{h,tt}(s)\|_0^2.
\]
Thus, we do not avoid with this alternative procedure the appearance of the factor $\|S^v\|_2^{1/2}$ as on the right-hand side 
of \eqref{ref}.
}

Applying inequality \eqref{ref} together with the inverse inequality \eqref{inv} and \eqref{eq:grad_uh_infty}, we obtain
\begin{equation}\label{eq:cotanablaPlinf}
C_{1,\rm inf}:=\max_{0\le j\le M}\|\nabla P_r^v \bu_h^j\|_\infty\le C_{\bu,1,{\rm inf}}+c_{\rm inv} h^{-d/2}\|S^v\|_2^{1/2}C_{L^2},
\end{equation}
and then, as before,
\begin{eqnarray}
\label{eq:cotanablaPlinf_sun}
&&K_{1,{\rm inf}}:=\Delta t\sum_{j=0}^{M}\|\nabla P_r^v \bu_h^j\|_\infty \le T C_{1,\rm inf}.
\end{eqnarray}
 Finally, arguing in the same way but applying \eqref{eq:uh_2d_dmenosuno} instead of \eqref{eq:grad_uh_infty},
 we find
 \begin{equation}\label{eq:cotanablaPld}
C_{\rm ld}:=\max_{0\le j\le M}\|\nabla P_r^v\bu^j\|_{L^{2d/(d-1)}} \le  C_{\bu,{\rm ld}}+c_{\rm inv}h^{-1/2}\|S^v\|_2^{1/2}C_{L^2}.
\end{equation}

The case
$
X=H_0^1(\Omega)^d 
$
is simpler. As before,  from  \eqref{eq:bound_2nd_term}, we get for $j=0,\ldots,M$,
\[\|\nabla(\bu_h^j-P_r^v \bu_h^j)\|_0\le C_{H^1}.
\]
Applying Poincar\'e's inequality yields
\[
\|\bu_h^j-P_r^v \bu_h^j\|_0\le C_pC_{H^1}.
\]
Arguing as before, we obtain 
\begin{eqnarray}\label{eq:cotaPlinf2_1}
 C_{\rm inf}:=\max_{0\le j\le M}\|P_r^v \bu_h^j\|_\infty &\le&
C_{\bu,{\rm inf}}+c_{\rm inv} C_p h^{-d/2}C_{H^1}, \\
\label{eq:cotanablaPlinf_1}
C_{1,\rm inf}:=\max_{0\le j\le M}\|\nabla P_r^v \bu_h^j\|_\infty&\le& C_{\bu,1,{\rm inf}}+c_{\rm inv} h^{-d/2}C_{H^1}, \\ 
\label{eq:cotanablaPld_1}
C_{\rm ld}:=\max_{0\le j\le M}\|\nabla P_r^v\bu^j\|_{L^{2d/(d-1)}}&\le& C_{\bu,{\rm ld}}+c_{\rm inv}h^{-1/2}C_{H^1}.
\end{eqnarray}
From \eqref{eq:cotaPlinf2_1}, it follows that 
\begin{equation}
\label{eq:cotaPlinf2_1_sum}
 K_{\rm inf}:=\Delta t\sum_{j=0}^{M}\|P_r^v \bu_h^j\|_\infty^2 \le T C_{\rm inf}^2
 \end{equation}
and from \eqref{eq:cotanablaPlinf_1} that 
\begin{eqnarray}
\label{eq:cotanablaPlinf_1_2}
K_{1,{\rm inf}}&:=&
\Delta t\sum_{j=0}^{M}\|\nabla P_r^v \bu_h^j\|_\infty \le T C_{1,\rm inf}.
\end{eqnarray}

\begin{remark}\rm
Let us observe that the factor $\|S^v\|_2^{1/2}$ appearing in  \eqref{eq:cotanablaPlinf}, \eqref{eq:cotanablaPlinf_sun}, and 
\eqref{eq:cotanablaPld} does not appear in \eqref{eq:cotanablaPlinf_1}, \eqref{eq:cotanablaPld_1}, and 
\eqref{eq:cotanablaPlinf_1_2}. Hence,  for a comparable value of $\sqrt{\lambda_{r+1}}$ in the first and the second bounds, the
second bounds, i.e.,  those corresponding to the case $X=H_0^1(\Omega)^d$, are smaller. 
\hspace*{\fill}$\Box$\end{remark}
{
\begin{remark}\label{re:bou}
The a priori bounds for $P_r^v \bu_h^j$ depend on the number and size of the eigenvalues and the mesh size $h$ and step size $\Delta t$. Notice, however, that given~$h$,  $r$, and $\Delta t$ can always be chosen so  that all the constants in the a priori bounds remain bounded as $h\rightarrow 0$. 
\hspace*{\fill}$\Box$\end{remark}
}

\section{The POD-ROM method}\label{sec:pod_rom}

\subsection{The case $X=H_0^1(\Omega)^d$}

We now consider the  grad-div POD-ROM model. For the sake of simplifying the error analysis, we use
the implicit Euler method as time integrator: For $n\ge 1$, find $\bu_r^n\in {\cal \bU}^r$ such that
for all $\bvar\in {\cal \bU}^r$
\begin{equation}\label{eq:pod_method2}
\left(\frac{\bu_r^{n}-\bu_r^{n-1}}{\Delta t},\bvar\right)+\nu(\nabla \bu_r^n,\nabla\bvar)+b_h(\bu_r^n,\bu_r^n,\bvar)
+\mu(\nabla \cdot\bu_r^n,\nabla \cdot\bvar)
=(\bff^{n},\bvar).
\end{equation}

Taking $t=t_j$ in \eqref{eq:gal_grad_div} and considering the second equation we get $\bu_h^j\in{\bV_h^l}$. 
Differentiating 
the second equation in \eqref{eq:gal_grad_div} with respect to $t$ and taking again $t=t_j$, we also
obtain $\bu_{h,t}^j\in{\bV_h^l}$.
As a consequence, we observe that ${\cal \bU}^r\subset {\bV_h^l}$ so that $\bu_r^n$ belongs to the 
space ${\bV_h^l}$ of discretely divergence-free functions.

Choosing $t=t_n$ in \eqref{eq:gal_grad_div2} yields for all $\bv_h\in {\bV}_{h,l}$
\begin{equation}\label{eq:gal_grad_div2_n}
\left(\bu_{h,t}^n,\bv_h\right)+\nu(\nabla \bu_h^n,\nabla \bv_h)+b_h(\bu_h^n,\bu_h^n,\bv_h)
+ \mu(\nabla \cdot\bu_h^n,\nabla \cdot \bv_h)
=({\boldsymbol f^n},\bv_h).
\end{equation}
Using the notation 
$
\bbeta_h^n=P_r^v\bu_h^n-\bu_h^n$, 
a straightforward calculation gives
\begin{eqnarray}\label{eq:prov_prop2}
\lefteqn{\hspace*{-5.1em}
\left(\frac{P_r^v \bu^{n}_h-P_r^v \bu^{n-1}_h}{\Delta t},\bvar\right)+\nu(\nabla  P_r^v \bu^n_h,\nabla\bvar)+b_h(P_r^v \bu^n_h,P_r^v \bu^n_h,\bvar)+\mu(\nabla \cdot P_r^v \bu^n_h,\nabla \cdot\bvar)}
\nonumber\\
&=&(\bff^{n},\bvar)+\left(\frac{P_r^v \bu^{n}_h-P_r^v \bu^{n-1}_h}{\Delta t}-\bu_{h,t}^n,\bvar\right)+\mu(\nabla \cdot\bbeta_h^n,\nabla \cdot \bvar)\\
&&+b_h(P_r^v\bu^n_h,P_r^v \bu^n_h,\bvar)-b_h(\bu^{n}_h,\bu^{n}_h,\bvar)\quad \forall\ \bvar\in {\cal \bU}^r.\nonumber
\end{eqnarray}
\begin{theorem}\label{thm:bound_L2L2}{
Let $(\bu,p)$ be the solution of the Navier--Stokes equations \eqref{NS}, 
which is assumed to be sufficiently regular, 
let $\bu_r$ be the grad-div POD
stabilized approximation defined in \eqref{eq:pod_method2} and assume
that the time step restriction \eqref{eq:time} holds. Then, the
following bound is valid
\begin{eqnarray}\label{eq:cota_finalSUPv}
\lefteqn{\sum_{j=1}^n\Delta t \|\bu_r^j-\bu^j\|_0^2}\nonumber\\ 
&\le& 3Te^{2C_u}\left[\|\be_r^0\|_0^2+
\left(2T(\mu+C_m^2T)(3+6(T/\tau)^2)+4C_p^2(T/\tau)^2\right)\sum_{k=r+1}^{d_v} \lambda_k\right.\nonumber\\
&& \left.{}+\left(CC_p^2T+2T(\mu+C_m^2T)\frac{16T}{3}+\frac{16T^2}{3}C_p^2\right)(\Delta t)^2\int_0^T\|\nabla(\bu_{h,tt})\|_0^2\ ds\right]\\
&& +3TC(\bu,p,l)^2h^{2l}+3TC_p^2 \left(3+6(T/\tau)^2\right)\sum_{k=r+1}^{d_v}\lambda_k.\nonumber
\end{eqnarray}}
\end{theorem}
\begin{proof}
Subtracting \eqref{eq:prov_prop2} from \eqref{eq:pod_method2} and denoting   by
$\be_r^n=\bu_r^n-P_r^v\bu_h^n \in {\cal \bU}^r$
leads to 
\begin{eqnarray}\label{eq:error_need}
\lefteqn{\left(\frac{\be_r^n -\be_r^{n-1}}{\Delta t},\bvar\right)+\nu(\nabla  \be_r^n ,\nabla\bvar)+\mu(\nabla \cdot \be_r^n,\nabla \cdot\bvar)}
\nonumber\\
&&+b_h(\bu_r^n,\bu_r^n,\bvar)-b_h(P_r^v  \bu^n_h,P_r^v \bu^n_h,\bvar)\nonumber\\
&=&\left(\bu_{h,t}^n-\frac{P_r^v \bu^{n}_h-P_r^v \bu^{n-1}_h}{\Delta t},\bvar\right)
 -\mu(\nabla \cdot\bbeta_h^n,\nabla \cdot \bvar)\nonumber\\
 &&+b_h(\bu^{n}_h,\bu^{n}_h,\bvar)-b_h(P_r^v\bu_r^n,P_r^v\bu_r^n,\bvar),\quad \forall\ \bvar\in {\cal \bU}^r.
\end{eqnarray}
Taking now $\bvar=\be_r^n$ yields
\begin{eqnarray}\label{eq:error2}
\lefteqn{\frac{1}{2\Delta t}\left(\|\be_r^n\|_0^2-\|\be_r^{n-1}\|_0^2\right)+\nu\|\nabla \be_r^n\|_0^2
+\mu\|\nabla \cdot \be_r^n\|_0^2}\nonumber\\
&\le& \left(\bu_{h,t}^n-\frac{P_r^v \bu^{n}_h-P_r^v \bu^{n-1}_h}{\Delta t},\be_r^n\right)+\big(b_h(P_r^v  \bu^n_h,P_r^v \bu^n_h,\be_r^n)-b_h(\bu_r^n,\bu_r^n,\be_r^n)\big)
\nonumber\\
&&-\mu(\nabla \cdot\bbeta_h^n,\nabla \cdot \be_r^n)+\big(b_h(\bu^{n}_h,\bu^{n}_h,\be_r^n)-b_h(P_r^v\bu_r^n,P_r^v\bu_r^n,\be_r^n)\big)
\nonumber\\
&=& I+II+III+IV.
\end{eqnarray}
The first term on the right-hand side of \eqref{eq:error2} is estimated by using the Cauchy--Schwarz 
and Young inequalities
\begin{eqnarray}\label{eq:first}
|I|\le \frac{T}{2}\left\|\bu_{h,t}^n-\frac{P_r^v \bu^{n}_h-P_r^v \bu^{n-1}_h}{\Delta t}\right\|_0^2+\frac{1}{2T}\|\be_r^n\|_0^2.
\end{eqnarray}
To bound the second term on the right-hand side of \eqref{eq:error2}, we use,  
following \cite[Eq.~(45)]{novo_rubino}, the skew-symmetry of the trilinear term, H\"older's inequality, 
\eqref{eq:cotaPlinf2_1}, \eqref{eq:cotanablaPlinf_1}, and Young's inequality to obtain
\begin{eqnarray}\label{eq:cota_er_1}
|II|= |b_h(\be_r^n,P_r^v \bu_h^n,\be_r^n)|&\le& \|\nabla P_r^v\bu_h^n\|_\infty\|\be_r^n\|_0^2+\frac{1}{2}\|\nabla \cdot \be_r^n\|_0
\|P_r^v\bu_h^n\|_\infty\|\be_r^n\|_0\nonumber\\
&\le& \left(\|\nabla P_r^v\bu_h^n\|_\infty+\frac{\|P_r^v\bu_h^n\|_\infty^2}{4\mu}\right)\|\be_r^n\|_0^2
+\frac{\mu}{4}\|\nabla \cdot \be_r^n\|_0^2.
\end{eqnarray}
For the third term, the application of the Cauchy--Schwarz and Young inequalities leads to 
\begin{eqnarray}\label{eq:third}
|III|\le \mu\|\nabla \bbeta_h^n\|_0^2+\frac{\mu}{4}\|\nabla \cdot \be_r^n\|_0^2.
\end{eqnarray}
For estimating the fourth term, we follow \cite[Eq.~(50)]{novo_rubino}.
Using H\"older's inequality, \eqref{eq:uh_infty}, \eqref{eq:uh_2d_dmenosuno}, 
the Sobolev imbedding \eqref{sob1} with $s=1$ and $q=2$, 
\eqref{diver_vol}, and Young's inequality leads to 
\begin{eqnarray}\label{eq:cota_er_6}
|IV|&\le& |b_h(P_r^v\bu_h^n,\bbeta_h^n,\be_r^n)|+|b_h(\bbeta_h^n,\bu_h^n,\be_r^n)|
\nonumber\\&\le&
\|P_r^v\bu_h^n\|_\infty\|\nabla \bbeta_h^n\|_0\|\be_r^n\|_0+\frac{1}{2}\|\nabla \cdot P_r^v\bu_h^n\|_{L^{2d/(d-1)}}
\| \bbeta_h^n\|_{L^{2d}}\|\be_r^n\|_0\nonumber\\
&&+\|\bbeta_h^n\|_{L^{2d}} \|\nabla  \bu_h^n\|_{L^{2d/(d-1)}}\|\be_r^n\|_0+\frac{1}{2}
\|\nabla \cdot\bbeta_h^n\|_0\|\bu_h^n\|_\infty\|\be_r^n\|_0\nonumber\\
&\le& C_{\rm inf}\|\nabla \bbeta_h^n\|_0\|\be_r^n\|_0+ C C_{\rm ld}\|\nabla \bbeta_h^n\|_0\|\be_r^n\|_0
\nonumber\\
&&+C\|\nabla \bbeta_h^n\|_0C_{\bu,\rm ld}\|\be_r^n\|_0+\frac{1}{2}\|\nabla \bbeta_h^n\|_0C_{\bu,\rm inf}\|\be_r^n\|_0\nonumber\\
&\le& C_m^2T\|\nabla \bbeta_h^n\|_0^2+\frac{1}{2T}\|\be_r^n\|_0^2,
\end{eqnarray}
where
\begin{equation}
\label{eq:lacm}
C_m=C(C_{\inf}+C_{\bu,\rm inf} + C_{\rm ld}+C_{\bu,\rm ld}).
\end{equation}
Inserting \eqref{eq:first}, \eqref{eq:cota_er_1}, \eqref{eq:third}, and \eqref{eq:cota_er_6} into \eqref{eq:error2}
and adding over the time instances, 
we get
\begin{eqnarray}\label{eq:error2a}
\lefteqn{\|\be_r^n\|_0^2+2\nu\sum_{j=1}^n\Delta t \|\nabla \be_r^j\|_0^2
+\mu\sum_{j=1}^n\Delta t\|\nabla \cdot \be_r^j\|_0^2}\nonumber\\
&\le&\|\be_r^0\|_0^2+\sum_{j=1}^n\Delta t \left(2\|\nabla P_r^v\bu_h^n\|_\infty+\frac{\|P_r^v\bu_h^n\|_\infty^2}{2\mu}+\frac{2}{T}\right)\|\be_r^j\|_0^2
\nonumber\\
&&+2(\mu+C_m^2T)\sum_{j=1}^n\Delta t\|\nabla \bbeta_h^j\|_0^2+T\sum_{j=1}^n\Delta t \left\|\bu_{h,t}^j-\frac{P_r^v \bu^{j}_h-P_r^v \bu^{j-1}_h}{\Delta t}\right\|_0^2.
\end{eqnarray}
Adding $\pm P_r^v \bu_{h,t}^j$ for bounding 
the last term on the right-hand side of \eqref{eq:error2a} gives
\begin{equation}\label{eq:uno}
\bu_{h,t}^j-\frac{P_r^v \bu^{j}_h-P_r^v \bu^{j-1}_h}{\Delta t}
= (\bu_{h,t}^j-P_r^v \bu_{h,t}^j)+\left(P_r^v \bu_{h,t}^j-\frac{P_r^v \bu^{j}_h-P_r^v \bu^{j-1}_h}{\Delta t}\right),
\end{equation}
from what follows 
\begin{eqnarray}\label{eq:last}
\lefteqn{
\sum_{j=1}^n\Delta t\left\|\bu_{h,t}^j-\frac{P_r^v \bu^{j}_h-P_r^v \bu^{j-1}_h}{\Delta t}\right\|_0^2}\\
&\le& 2
\sum_{j=1}^n\Delta t\left\|P_r^v \bu_{h,t}^j-\bu_{h,t}^j\right\|_0^2
+2\sum_{j=1}^n\Delta t\left\|P_r^v\left(\bu_{h,t}^j-\frac{\bu^{j}_h- \bu^{j-1}_h}{\Delta t}\right)\right\|_0^2.\nonumber
\end{eqnarray}
To bound the first term on the right-hand side of \eqref{eq:last}, we apply Poincar\'e's inequality \eqref{poincare} and \eqref{eq:cota_pod_deriv}, taking into account
that $(M+1)/M\le 2$, to get
\begin{equation}\label{eq:unob}
\sum_{j=1}^n\Delta t\left\|P_r^v \bu_{h,t}^j-\bu_{h,t}^j\right\|_0^2
\le C_p^2\sum_{j=1}^n\Delta t\left\|\nabla\left(P_r^v \bu_{h,t}^j-\bu_{h,t}^j\right)\right\|_0^2
\le \frac{2C_p^2T}{\tau^2} \sum_{k=r+1}^{d_v} \lambda_k.
\end{equation}
For the second one, we also use Poincar\'e's inequality \eqref{poincare} and then notice that, 
utilizing a property of a projection in Hilbert spaces,  
$\|\nabla P_r^v\bw\|_0^2\le \|\nabla \bw\|_0^2$ for every $\bw\in H_0^1(\Omega)^d$, so that
\begin{equation}\label{eq:unobb}
\sum_{j=1}^n\Delta t\left\|P_r^v\left(\bu_{h,t}^j-\frac{\bu^{j}_h- \bu^{j-1}_h}{\Delta t}\right)\right\|_0^2
\le C_p^2 \sum_{j=1}^n\Delta t\left\|\nabla \left(\bu_{h,t}^j-\frac{\bu_h^j-\bu_h^{j-1}}{\Delta t}\right)\right\|_0^2.
\end{equation}
Inserting \eqref{eq:unob} and \eqref{eq:unobb} into \eqref{eq:last}, we obtain
\begin{eqnarray}\label{eq:lastb}
\lefteqn{\sum_{j=1}^n\Delta t\left\|\bu_{h,t}^j-\frac{P_r^v \bu^{j}_h-P_r^v \bu^{j-1}_h}{\Delta t}\right\|_0^2}
\nonumber\\
&\le& \frac{2C_p^2T}{\tau^2} \sum_{k=r+1}^{d_v} \lambda_k +C_p^2\sum_{j=1}^n\Delta t\left\| \nabla\left(\bu_{h,t}^j-\frac{\bu^{j}_h- \bu^{j-1}_h}{\Delta t}\right)\right\|_0^2.
\end{eqnarray}
For the second term on the right-hand side of \eqref{eq:lastb}, a standard argument gives {(e.g., see \cite[(81)]{NS_grad_div})}
\begin{equation}\label{eq:est_temp_der_grad}
C_p^2\sum_{j=1}^n\Delta t\left\| \nabla\left(\bu_{h,t}^j-\frac{\bu^{j}_h- \bu^{j-1}_h}{\Delta t}\right)\right\|_0^2
\le CC_p^2(\Delta t)^2\int_0^T\|\nabla(\bu_{h,tt})\|_0^2\ ds.
\end{equation}

Assuming that 
\begin{equation}\label{eq:time}
\Delta t \left(2C_{1,\rm inf}+\frac{C_{\rm inf}^2}{2\mu}+\frac{2}{T}\right)\le \frac{1}{2},
\end{equation}
 denoting by
\begin{equation}\label{eq:C_u}
C_u=2K_{1,{\rm inf}}+\frac{K_{\rm inf}^2}{2\mu}+2,
\end{equation}
applying Gronwall's Lemma \cite[Lemma~5.1]{hey_ran_IV}, 
\eqref{eq:cota_pod_0}, and taking into account {again} that $(M+1)/M\le 2$, we obtain from \eqref{eq:error2a}
\begin{eqnarray}\label{eq:error3_b}
\lefteqn{
\|\be_r^n\|_0^2+2\nu\sum_{j=1}^n\Delta t \|\nabla \be_r^j\|_0^2+\mu\sum_{j=1}^n\Delta t\|\nabla \cdot \be_r^j\|_0^2} \\
&\le& e^{2C_u}\biggl(\|\be_r^0\|_0^2+\left(2T(\mu+C_m^2T)(3+6(T/\tau)^2)+4C_p^2(T/\tau)^2\right)\sum_{k=r+1}^{d_v}\lambda_k\
\nonumber\\
&&{}+\left(CC_p^2T+2T(\mu+C_m^2T)\frac{16T}{3}\right)(\Delta t)^2\int_0^T\|\nabla(\bu_{h,tt})\|_0^2\ ds\biggr).
\nonumber
\end{eqnarray}

We have $\sum_{j=1}^n\Delta t \|\be_r^j\|_0^2\le T \max_{1\le j\le n}\|\be_r^j\|_0^2$ and
\[
\sum_{j=1}^n\Delta t \|\bu_r^j-\bu^j\|_0^2
\le 3\left(\sum_{j=1}^n\Delta t \|\be_r^j\|_0^2+\sum_{j=1}^n\Delta t\|P_r^v \bu_h^j-\bu_h^j\|_0^2
+\sum_{j=1}^n\Delta t\| \bu_h^j-\bu^j\|_0^2\right).
\]
Inserting the estimates \eqref{eq:error3_b}, \eqref{eq:cota_pod_0}, and \eqref{eq:cota_grad_div}
leads directly to  \eqref{eq:cota_finalSUPv}.
\end{proof}
{\begin{remark}\rm 
Arguing as in \cite[Proposition~3.2]{hey_ran_IV}, one can get an a priori bound for 
\begin{equation}\label{second_dev}
\int_0^T\|\nabla(\bu_{h,tt})\|_0^2
\end{equation}
in \eqref{eq:cota_finalSUPv}.
 However, the
error analysis in \cite{hey_ran_IV} is not valid for high Reynolds numbers. In the appendix we get an error bound for this
term with constants independent of inverse powers of the viscosity.
A robust estimate for \eqref{second_dev}, in the case $l\ge3$, is derived.
Hence, for $l\ge3$, there is no explicit appearance of inverse powers of the 
viscosity coefficient in the error bound \eqref{eq:cota_finalSUPv}. The technical reason for not obtaining a robust estimate for $l=2$ with the technique from 
the appendix is the gradient in front of $\bu_{h,tt}$. This gradient was introduced with the transition 
from the $L^2(\Omega)^d$ norm to the corresponding norm of the gradient in  
\eqref{eq:unobb} in order to be able to apply the Hilbert space argument.  Note that, with the approach
presented in the appendix, boundedness can be shown also for $l=2$, but not the robustness of the bound. 
\hspace*{\fill}$\Box$\end{remark}
}
{\begin{remark} \rm In view of Remark \ref{re:bou} and assuming the solution is regular enough,
assumption \eqref{eq:time} only requires a step size smaller than a given constant whose size depends
on the constants $C_{1,\rm inf}$, $C_{\rm inf}$, $\mu$ and $T$. 
\hspace*{\fill}$\Box$\end{remark}
}

\begin{remark}\rm With the error decomposition in the proof of Theorem \ref{thm:bound_L2L2} and applying \eqref{eq:inv_M} or the inverse inequality \eqref{inv} to \eqref{eq:error3_b},
one can also prove a robust error bound for $\sum_{j=1}^n\Delta t \|\nabla (\bu_r^j-\bu^j)\|_0^2$.
\hspace*{\fill}$\Box$\end{remark}

We can apply Lemma~\ref{le:our_etal} to get pointwise estimates with respect to time 
both in $L^2(\Omega)$ and $H^1(\Omega)$. Let us prove pointwise estimates in $L^2(\Omega)$,
the argument for proving bounds in $H^1(\Omega)$ is the same. 
Since 
\[
 \|\bu_r^n-\bu^n\|_0^2\le
3\|\be_r^n\|_0^2+3\|P_r^v \bu_h^n-\bu_h^n\|_0^2+3\| \bu_h^n-\bu^n\|_0^2,
\]
we can utilize \eqref{eq:error3_b} and \eqref{eq:cota_grad_div} to bound the first and third term on the right-hand side. For bounding the second term, \eqref{eq:bound_2nd_term} is utilized. More precisely, taking into account that the case $X=H_0^1(\Omega)^d$ is analyzed, we
apply Poincar\'e's inequality \eqref{poincare}  to bound
$\|P_r^v \bu_h^n-\bu_h^n\|_0^2\le C_p^2 \|\nabla(P_r^v \bu_h^n-\bu_h^n)\|_0^2$ 
and then \eqref{eq:bound_2nd_term}.
Collecting the estimates \eqref{eq:error3_b}, \eqref{eq:bound_2nd_term}, \eqref{eq:cota_grad_div}
proves the following theorem. 

\begin{theorem} Let the assumptions of Theorem~\ref{thm:bound_L2L2} be satisfied, then the following bound is 
valid
\begin{eqnarray*}
\lefteqn{\max_{0\le n\le M}\|\bu_r^n-\bu^n\|_0^2} 
 \\
&\le& 3e^{2C_u}\left[\|\be_r^0\|_0^2+\left(2T(\mu+C_m^2T)(3+6(T/\tau)^2)+4C_p^2(T/\tau)^2\right)\sum_{k=r+1}^{d_v}\lambda_k\right.
\nonumber\\
&& +\left.\left(CC_p^2T+2T(\mu+C_m^2T)\frac{16T}{3}+\frac{16T}{3}C_p^2\right)(\Delta t)^2\int_0^T\|\nabla(\bu_{h,tt})\|_0^2\ ds\right]\\
&& +3C_p^2(3+6(T/\tau)^2)\sum_{k=r+1}^{d_v}\lambda_k
+ {C(\bu,p,l+1)^2 h^{2l}},
\end{eqnarray*}
where the constants for $l\ge3$ do not blow up for small viscosity coefficients since \eqref{second_dev}
is bounded.
\end{theorem}

\subsection{The case $X=L^2(\Omega)^d$}

For the sake of brevity, in the proof of the following theorem, we are going to mention only the differences with respect to the analysis for the case $X=H_0^1(\Omega)^d$.

{
\begin{theorem}\label{thm:bound_L2L20}
Assume that the solution $(\bu,p)$ of the
Navier--Stokes equations \eqref{NS} is sufficiently regular, that the time step restriction 
\eqref{eq:time} is satisfied, and let $\bu_r$ be the grad-div POD
stabilized approximation defined in \eqref{eq:pod_method2}.  Then, the
following bound holds.
\begin{eqnarray*}\label{eq:cota_finalSUPv0}
\lefteqn{\sum_{j=1}^n\Delta t \|\bu_r^j-\bu^j\|_0^2}\nonumber\\ 
&\le& 3Te^{2C_u}\left[\|\be_r^0\|_0^2+2T(\nu+\mu+C_m^2T)\|S^v\|_2(3+6(T/\tau)^2)\sum_{k=r+1}^{d_v} \lambda_k\right.\nonumber\\
&& \left.{}+(\Delta t)^2\left(CT+2T(\nu+\mu+C_m^2T)\|S^v\|_2 (16T^2)/3+(16T^2)/3\right)\int_0^T\|\bu_{h,tt}\|_0^2\ ds\right]\\
&& +3TC(\bu,p,l+1)^2h^{2l}+3T\left(3+6(T/\tau)^2\right)\sum_{k=r+1}^{d_v}\lambda_k.\nonumber
\end{eqnarray*}
\end{theorem}
}
\begin{proof}
Observing that the orthogonality property of $P_r^v$ affects now the term with the approximation 
of the temporal derivative (instead of the viscous term) we obtain, instead of \eqref{eq:error_need},
the following error equation
\begin{eqnarray}\label{eq:error_need_0}
\lefteqn{\left(\frac{\be_r^n -\be_r^{n-1}}{\Delta t},\bvar\right)+\nu(\nabla  \be_r^n ,\nabla\bvar)+\mu(\nabla \cdot \be_r^n,\nabla \cdot\bvar)}\nonumber\\
&& +b_h(\bu_r^n,\bu_r^n,\bvar)-b_h(P_r^v  \bu^n_h,P_r^v \bu^n_h,\bvar)\nonumber\\
&=&\left(\bu_{h,t}^n-\frac{\bu_h^n-\bu_h^{n-1}}{\Delta t },\bvar\right)-\nu(\nabla \bbeta_h^n,\nabla  \bvar)
 -\mu(\nabla \cdot\bbeta_h^n,\nabla \cdot \bvar)\nonumber\\
 &&+b_h(\bu^{n}_h,\bu^{n}_h,\bvar)-b_h(P_r^v  \bu^n_h,P_r^v \bu^n_h,\bvar)\quad \forall\ \bvar\in {\cal \bU}^r.
\end{eqnarray}
Using the same techniques as for the other case, we infer, instead of \eqref{eq:error2a}, that 
\begin{eqnarray}\label{eq:error2_0}
\lefteqn{\|\be_r^n\|_0^2+2\nu\sum_{j=1}^n\Delta t \|\nabla \be_r^j\|_0^2
+\mu\sum_{j=1}^n\Delta t\|\nabla \cdot \be_r^j\|_0^2}\nonumber\\
&\le& \|\be_r^0\|_0^2+\sum_{j=1}^n\Delta t\left(2\|\nabla P_r^v\bu_h^n\|_\infty+\frac{\|P_r^v\bu_h^n\|_\infty^2}{2\mu}+\frac{2}{T}\right)\|\be_r^j\|_0^2\nonumber\\
&&+2(\nu +\mu+C_m^2T)\sum_{j=1}^n\Delta t\|\nabla \bbeta_h^j\|_0^2 
+CT(\Delta t)^2\int_0^T\|\bu_{h,tt}(s)\|_0^2\ ds.
\end{eqnarray}
From \eqref{eq:error2_0} we continue as before, compare \cite[Theorem~5.3]{novo_rubino}, and 
take into account that applying \eqref{eq:inv_S} and \eqref{eq:cota_pod_0}, it is 
\[
\sum_{j=1}^n\Delta t\|\nabla \bbeta_h^j\|_0^2 \le T \|S^v\|_2\left((3+6(T/\tau)^2)\sum_{k=r+1}^{d_v}\lambda_k+
\frac{16T^2}{3}(\Delta t)^2\int_0^T \|\bu_{h,tt}(s)\|_0^2\ ds\right).
\]
Then, instead of \eqref{eq:error3_b}, we conclude
\begin{eqnarray}\label{eq:error3_b_L2}
\lefteqn{
\|\be_r^n\|_0^2+2\nu\sum_{j=1}^n\Delta t \|\nabla \be_r^j\|_0^2+\mu\sum_{j=1}^n\Delta t\|\nabla \cdot \be_r^j\|_0^2} \nonumber\\
&\le& e^{2C_u}\biggl(\|\be_r^0\|_0^2+2T(\nu+\mu+C_m^2T)\|S^v\|_2(3+6(T/\tau)^2)\sum_{k=r+1}^{d_v}\lambda_k\
\nonumber\\
&&{}+(\Delta t)^2\left(CT+2T(\nu+\mu+C_m^2T)\|S^v\|_2 (16T^2)/3\right)\int_0^T\|\bu_{h,tt}(s)\|_0^2\ ds\biggr).
\end{eqnarray}
\end{proof}
\begin{remark}\rm
A robust estimate for the following term and $l\ge2$ is proved in the appendix
\begin{equation}\label{second_dev_2}
\int_0^T\|\bu_{h,tt}\|_0^2\ ds.
\end{equation} 

In \cite[Section~5]{novo_rubino}, there is an error analysis for the case in which a fully discrete Galerkin method 
with the same implicit Euler method is utilized for the FOM.
\hspace*{\fill}$\Box$\end{remark}
{
\begin{remark}\rm For the proof of Theorem \ref{thm:bound_L2L20} one can compare $\bu_r^n$ with the Ritz projection $R_r^v \bu_h^n$ 
defined by
\[
(\nabla R_r^v \bu_h^n,\nabla \bvar_i)=(\nabla \bu_h^n,\nabla\bvar_i),\quad i=1,\ldots,r,
\]
instead of the $L^2$ projection $P_r^v \bu_h^n$, as suggested in \cite[Lemma 4.3]{koc_rubino_et_al}. However, unlike \cite[Lemma 4.3]{koc_rubino_et_al}, in which optimal bounds are obtained for the heat equation, the comparison with the Ritz projection does not allow to improve the error bounds for the Navier--Stokes equations. On the one hand, although comparing with $R_r^v \bu_h^n$ the error term $\nu\|\nabla (R_r^v \bu_h^n-\bu_h^n)\|_0$ disappears (this term is $\nu\|\nabla \bbeta_h^n\|_0$ in Theorem~\ref{thm:bound_L2L20}), we cannot get rid of the term
$\mu\|\nabla \cdot(R_r^v \bu_h^n-\bu_h^n)\|_0$ (this term is $\mu\|\nabla \cdot\bbeta_h^n\|_0$ in Theorem~\ref{thm:bound_L2L20}). On the other hand, there is also a term involving $\|\nabla (R_r^v \bu_h^n-\bu_h^n)\|_0$ coming from the nonlinear term ($\|\nabla \bbeta_h^n\|_0$ in Theorem \ref{thm:bound_L2L20}) that does not disappear either. Thus, the $H_0^1(\Omega)^d$ norm of the error in the Ritz projection has to be estimated anyway.
\hspace*{\fill}$\Box$\end{remark}
}

\begin{remark}\rm
Let us observe that at this point it seems that the snapshots from the time derivative are not helpful in the $L^2(\Omega)^d$ case.
However, it turns out that these snapshots are needed for obtaining pointwise estimates in time. 
In fact, we can repeat the arguments of the previous section to get such estimates for
the $L^2(\Omega)^d$ error, and also for the $H^1(\Omega)^d$ error applying the inverse inequality \eqref{eq:inv_S}, that cannot be
obtained, at least with the same arguments, without adding those snapshots.
\hspace*{\fill}$\Box$\end{remark}

\begin{remark}\rm {As observed already in \cite{kunisch,KV_2002}}, comparing the $L^2(\Omega)^d$ and the $H_0^1(\Omega)^d$ cases, we can observe that in the $L^2(\Omega)^d$ case the eigenvalues
are multiplied by $\|S^v\|_2$ in the error bound \eqref{eq:error3_b_L2}, which increases the size of this term, 
compared with \eqref{eq:error3_b}.
In addition, the factor  $\|S^v\|_2^{1/2}$ appears in the definition \eqref{eq:C_u} of $C_u$, 
which
gives a bigger constant in the exponential term of the error bound, and also in the assumption for the time steps \eqref{eq:time}, i.e.,
the condition for applying Gronwall's lemma leads to a severe time 
step restriction if $\|S^v\|_2$ is very large. 
On the other hand, using $L^2(\Omega)^d$ as projection space leads to an integral term in the error
bound, \eqref{second_dev_2}, that can be bounded in a robust way for $l\ge2$, whereas $l\ge3$ is necessary if 
the projection space is 
$H_0^1(\Omega)^d$,  to bound the term \eqref{second_dev}, at least with the approach from the appendix. 
Thus, both approaches possess advantages with respect to certain aspects of the error analysis. 
\hspace*{\fill}$\Box$\end{remark}

\begin{remark}\rm A popular pair of finite element spaces that leads to weakly di\-ver\-gence-free discrete
velocity fields is the Scott--Vogelius pair $({\boldsymbol X}_h^l, Q_{h,\mathrm{disc}}^{l-1})$. It is inf-sup 
stable for $l\ge d$ on so-called barycentric-refined grids \cite{Zha05}. A favorable feature of the 
Scott--Vogelius pair is that it leads to so-called pressure-robust velocity estimates, i.e., the 
velocity error bounds do not depend on the pressure. In particular, an estimate of form \eqref{eq:cota_grad_div} was derived in \cite{SL17} with $C(\bu,l+1)$. In estimating the integral term 
as performed in the appendix, the term $(\sigma_2,\nabla\cdot \bvar_h)$ vanishes if $ \bvar_h$ is 
weakly divergence-free, such that the bound of the integral term does not contain the pressure. 
Finally, the snapshots are only from the velocity field, which can be (formally) computed by 
solving an equation of type \eqref{eq:gal_grad_div2}, which does not contain the pressure. 
Hence all terms in estimates  \eqref{eq:cota_pod_0} and \eqref{eq:cota_pod_deriv} are independent of the 
pressure. Applying the same analysis as for the Taylor--Hood pair of spaces, even with some 
simplifications, gives for the Scott--Vogelius pair pressure-robust error bounds of the same type, 
under the conditions on the inf-sup stability mentioned above. 
\hspace*{\fill}$\Box$\end{remark}


\section{Numerical studies}\label{sec:num}
We now present some results for the well-known benchmark problem defined in~\cite{bench}. The domain is given by
\[
\Omega=(0,2.2)\times(0,0.41)/\left\{ (x,y) \mid (x-0.2)^2 + (y-0.2)^2 \le 0.0025\right\}.
\]
On the inflow boundary, {$x=0$}, the velocity is  prescribed by
\[
\bu(0,y)=\frac{6}{0.41^2}\sin\left(\frac{\pi t}{8}\right)\left(\begin{array}{c} y(0.41-y)\\ 0\end{array}\right)
\]
and on the outflow boundary $x=2.2$ we set the so-called ``do nothing" boundary condition.
On the rest of the boundary the velocity is set $\bu ={\bf 0}$. It is well known that for $\nu=0.001$ there is a stable periodic orbit.

We use piecewise quadratic and linear elements for velocity and pressure, respectively, on the same mesh as in~\cite{BDF2},
resulting in 27168 degrees of freedom for the velocity and 3480 for the pressure. We call this grid here the main
grid. The grid obtained from this one by regular refinement  (107328 degrees of freedom for the velocity and 13584 for the pressure), called refined grid henceforth, was used only to compute a reference solution and the errors of the snapshots. {For the time integration of the semidiscrete FOM~\eqref{eq:gal_grad_div} we used a variable stepsize  second order backward differentiation formula (BDF2) described in~\cite{BDF2}}. On both grids we computed the periodic orbit 
{(see \cite{periodic_orbit}, \cite[Example D.8]{John})} by finding the fixed point of the return map to a Poincar\'e section (e.g., see \cite{periodic}). The periods, computed with a relative error below $10^{-6}$ were $T=0.331761$ and~$T_r=0.331338$ on the main and refined grid, respectively. We computed snapshots over one period with a spacing of $T/1024$ and~$T_r/1024$. {Notice that velocity, acceleration and pressure at these equally-spaced time levels were computed by interpolation, since they do not necessarily coincide with the steps taken by the variable stepsize BDF2 code; in fact, for the snapshots on the main grid, the BDF2 code took 5553 steps, the stepsizes selected as explained in~\cite{BDF2}, when the tolerances for the absolute and relative values of the local errors were set to~$10^{-10}$ and~$10^{-7}$, respectively (see~\cite{BDF2} for details).} Taking as exact the results on the refined grid, the relative errors of the snapshots are shown in Fig.~\ref{err_snap}, both in velocity and in acceleration, and  both in the $L^2$ and~$H_0^1$ norms. It can be seen that the errors in velocity are around 0.003 in the $L^2$ norm and around 0.03 in the~$H_0^1$ norm, and that the acceleration errors are, approximately, eight times larger than those of the velocity in the $L^2$ norm, and four times larger in the~$H_0^1$ norm.

\begin{figure}[t!]
\centerline{\includegraphics[width=0.48\textwidth]{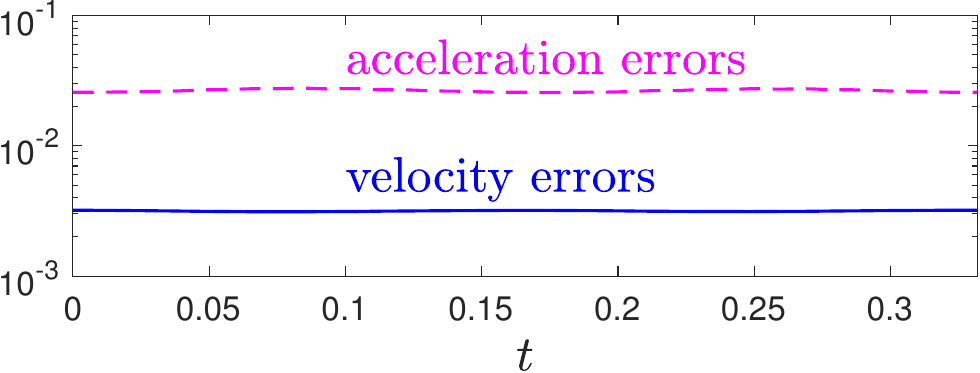}\hspace*{1em}
 \includegraphics[width=0.48\textwidth]{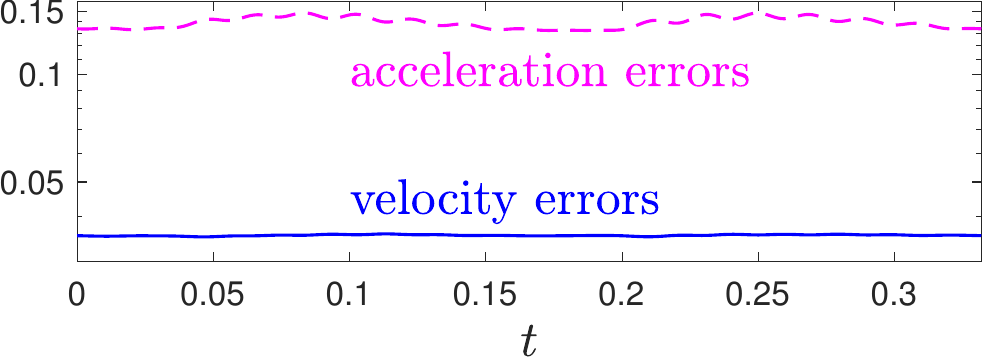}}
\caption{Snapshot errors in the $L^2$ norm (left) and in the $H_0^1$ norm (right).}\label{err_snap} 
\end{figure}

Fig.~\ref{sing_values} shows the first 256 singular values $\sigma_k=\sqrt{\lambda_k}$ relative to their Euclidean norm, that is,
\begin{equation}
\label{singular}
\sigma_k/\Bigg(\sum_{j=1}^N \sigma_k^2\Bigg)^{1/2},
\end{equation}
where $N$ is the number of snapshots, both when the inner product is that of $L^2$ and $H_0^1$, in the cases where the elements in the data sets are $\by_h^j=\bu_h^j-\overline \bu_h$, $\by_h^j=T\bu_{u,t}^j$, and $\by_h^j=T\delta_t\bu_h^j=T(\bu_h^j-\bu_h^{j-1})/\Delta t$.
\begin{figure}
\centerline{\includegraphics[width=0.48\textwidth]{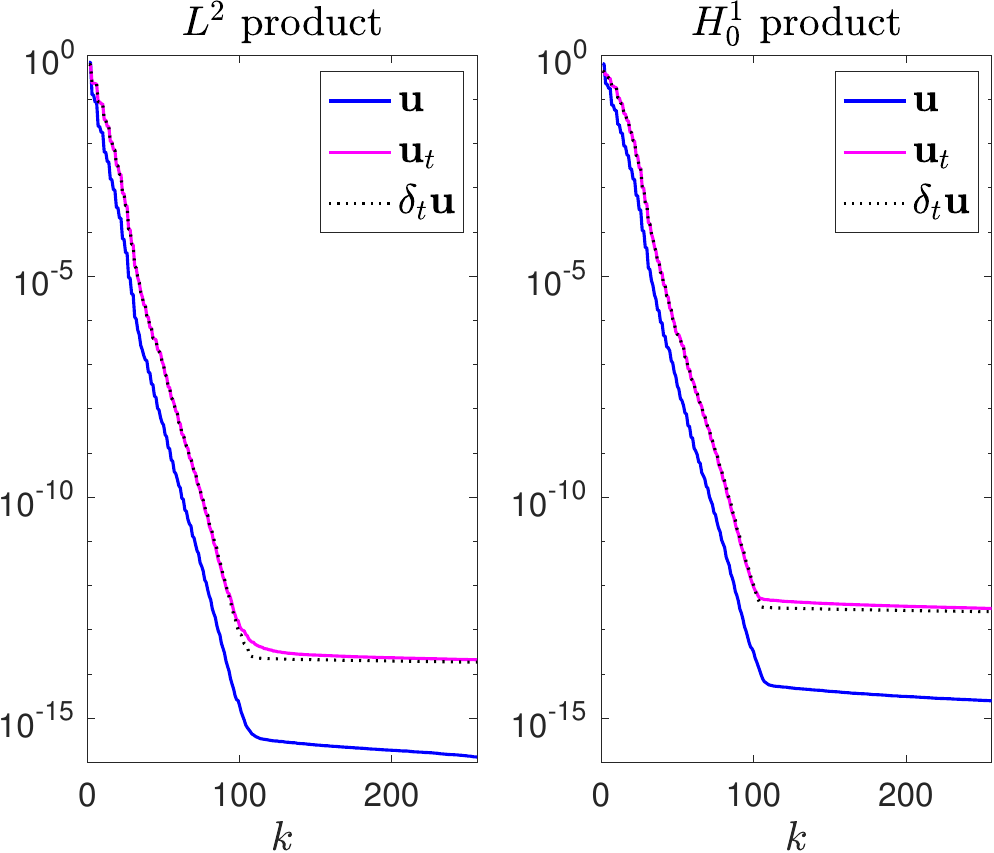}}
\caption{First 256 singular values of data set relative to their Euclidean norm~(\ref{singular}), for the $L^2$ inner product (left) and the $H_0^1$ inner product (right).
}\label{sing_values} 
\end{figure}
We see that with both inner products, the singular values are slightly larger 
when the data set is the time derivatives or their approximation, but already for $k=50$ their value is considerable smaller than the approximation errors in~Fig.~\ref{err_snap}. In fact, for the singular values corresponding to the data set given by the fluctuations $\bu_h^j-\overline \bu_h$ (blue line), there are only $16$ values above $10^{-3}$ in Fig.~\ref{sing_values} (left) and $14$ above~$10^{-2}$ in the right plot. Note that $10^{-3}$ and~$10^{-2}$ are about one third of the average errors in Fig.~\ref{err_snap} for the $L^2$ and the~$H_0^1$ norm, so it reasonable to use a POD basis with no more elements, since although with a larger basis we may better approximate the elements in the data set, these have already larger approximation errors. Consequently, in the sequel, we use a POD basis with 16 elements if the  $L^2$ inner product is utilized and 
$14$ in case of the $H_0^1$ inner product.

Next we present in~Fig.~\ref{proy_err} the projection errors of the snapshots,
\begin{equation}
\label{proy_err_f}
\left\| \bu_h^j - P_r^v \bu_h^j\right\|_i \bigg/  \left\| \bu_h^j \right\|_i,\quad j=1,\ldots,N, \quad i \in \{0,1\}, 
\end{equation}
for data sets $\by_h^j=\bu_h^j-\overline \bu_h$, $\by_h^j=T\bu_{u,t}^j$, and $\by_h^j=T\delta_t\bu_h^j={T}(\bu_h^j-\bu_h^{j-1})/\Delta t$, when $r=16$ and $r=14$ for POD basis  based on $L^2$ and $H_0^1$ inner products, respectively. 
\begin{figure}
\centerline{\includegraphics[width=0.48\textwidth]{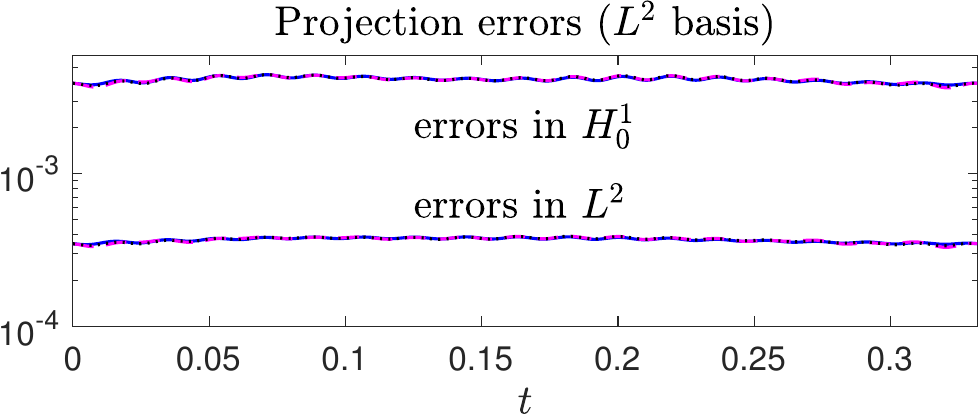}\hspace*{1em}
 \includegraphics[width=0.48\textwidth]{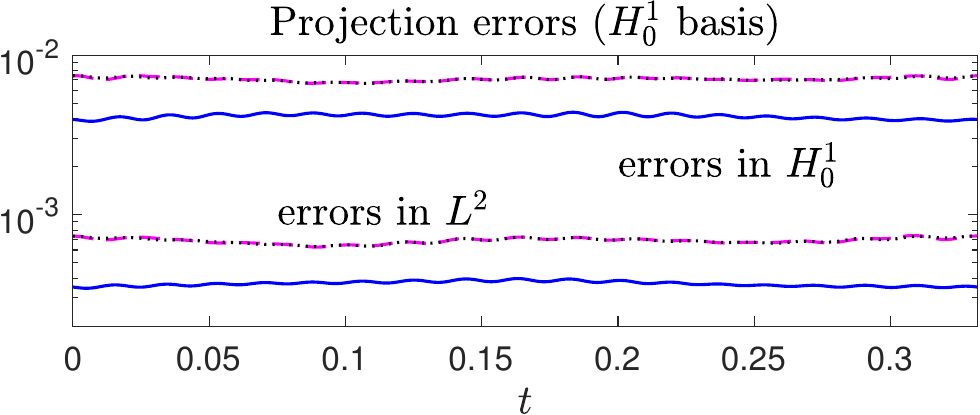}}
\caption{Snapshot  projection errors~\eqref{proy_err_f} for the correlation matrix based on $L^2$ product (left) and $H_0^1$ products (right). Results for data sets $\by_h^j=\bu_h^j-\overline \bu_h$ (blue, {continuous line}), $\by_h^j=T\bu_{u,t}^j$ (magenta, {discontinuous line}), and $\by_h^j=T\delta_t\bu_h^j={T}(\bu_h^j-\bu_h^{j-1})/\Delta t$ (black, {dotted line}).
{Note that several lines are on top of each other.}
}\label{proy_err} 
\end{figure}
We notice that the results on the left plot are quite independent of the data set used, whereas those on the right plot show some advantage for the data set based on the fluctuations, $\by_h^j=\bu_h^j-\overline \bu_h$, although in all cases the projection errors are well below  the quantity~\eqref{singular}
that we used to select the size of the POD bases,
which, recall, are $0.001$ and $0.01$ for POD basis based on $L^2$ and $H_0^1$, respectively.  

Fig.~\ref{pod-rom_err} depicts the POD-ROM errors
\begin{equation}
\label{pod-rom_err_f}
\left\|\bu_r^j -\bu_h^j \right\|_i  \bigg/ {\left\| \bu_h^j \right\|_i},\quad j=1,\ldots,N, \quad 
i \in \{0,1\},
\end{equation}
with the same color convention as in~Fig.~\ref{proy_err}. 
\begin{figure}
\centerline{\includegraphics[width=0.48\textwidth]{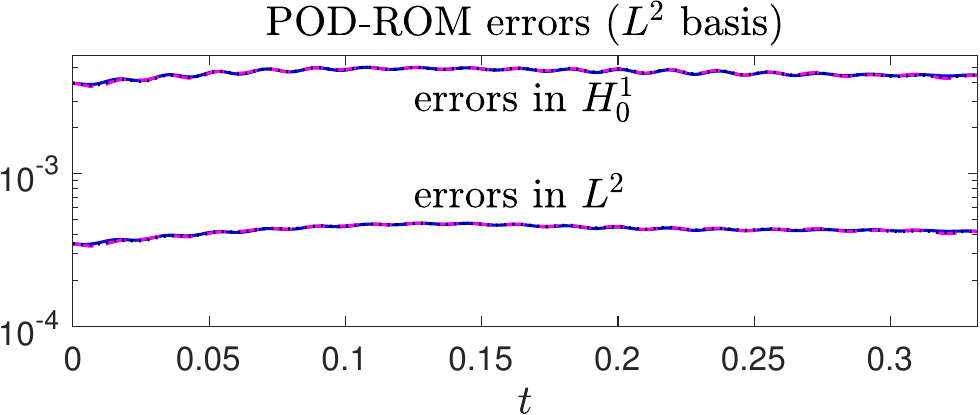}\hspace*{1em}
 \includegraphics[width=0.48\textwidth]{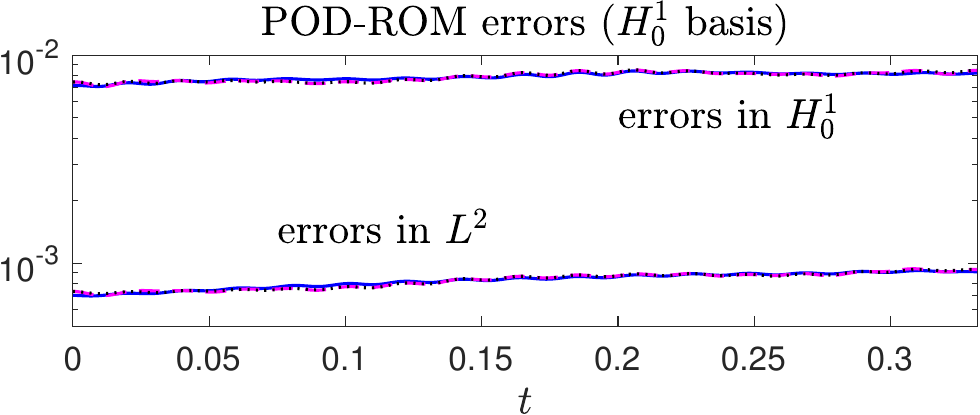}}
\caption{POD-ROM errors~\eqref{pod-rom_err_f} for correlation matrix based on $L^2$ product (left) and $H_0^1$ products (right). Results for data sets $\by_h^j=\bu_h^j-\overline \bu_h$ (blue, {continuous line}), $\by_h^j=T\bu_{u,t}^j$ (magenta, {discontinuous line}) and $\by_h^j=T\delta_t\bu_h^j={T}(\bu_h^j-\bu_h^{j-1})/\Delta t$ (black, {dotted line}).
}\label{pod-rom_err} 
\end{figure}
For the time integration, instead of the backward Euler method in~\eqref{eq:pod_method2}, we used the two step BDF formula with fixed step size $\Delta t=T/N$, $N=1024$, except for the first step, where the backward Euler method was applied. The parameter for the grad-div term was as in the computation of the snapshots $\mu=0.01$. 
We see that the results are very independent of the data set used, for both inner products, and that the POD-ROM errors are slightly larger than the corresponding projection errors, except for data sets based on snapshots, $\by_h^j=\bu_h^j-\overline \bu_h$ (blue line) and POD basis based on $H_0^1$ inner products (right plot) which are
approximately twice the size of the corresponding errors in~Fig.~\ref{proy_err}.

Finally, lift and drag coefficients are studied. For the FOM, they were computed as indicated in~\cite{John} (see also~\cite{BDF2}). For the POD-ROM method they were computed as in~\cite{pod_da_nos}. In Fig.~\ref{lifts-drags} we see the evolution of the drag and lift coefficients along 12 periods for the reference solution (computed on the refined grid), the FOM, and the POD-ROM approximation with the fluctuations $\bu_h^j-\overline\bu_h$ of the snapshots computed  in the first period as data set with the $L^2$ inner product (all other POD-ROM approximations gave similar results).
\begin{figure}
\centerline{\includegraphics[height=2.5truecm]{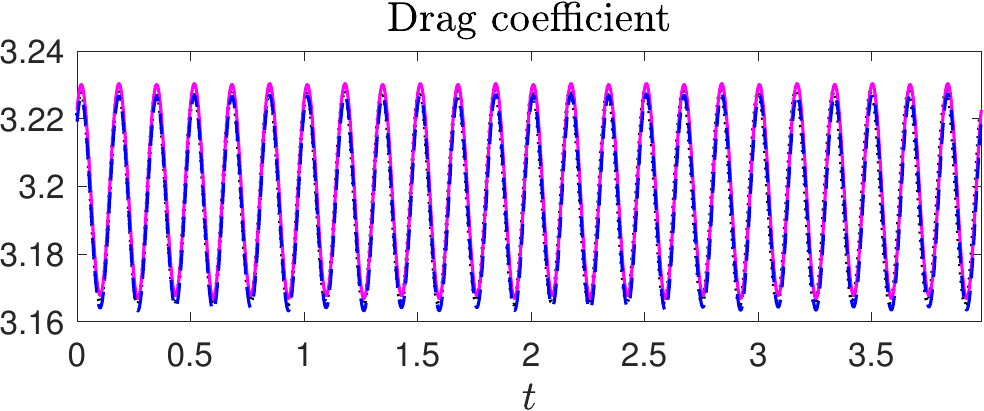}\hspace*{1em}
 \includegraphics[height=2.5truecm]{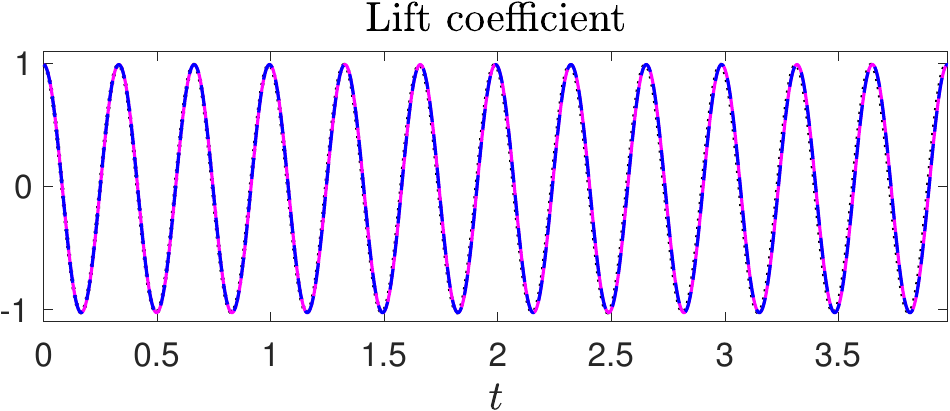}}
\caption{Drag (left) and lift (right) coefficients for the reference solution (black, {dotted line}), the FOM (magenta, {continuous line}), and the POD-ROM approximation with snapshots $\bu_h^j$ as data set and $L^2$ inner product {(blue, discontinuous line)}.
}\label{lifts-drags} 
\end{figure}
At first sight, the agreement is excellent. Fig.~\ref{detalles} compares the first and twelfth periods. There is a phase difference between the black {dotted} line and the other two lines, since, as commented above, the ``exact" solution computed on the refined grid has a slightly different period than that computed on the main grid.  We notice however, that the difference between the other two lines is hardly altered from the first period to the last one.
\begin{figure}
\centerline
{\includegraphics[height=2.5truecm]{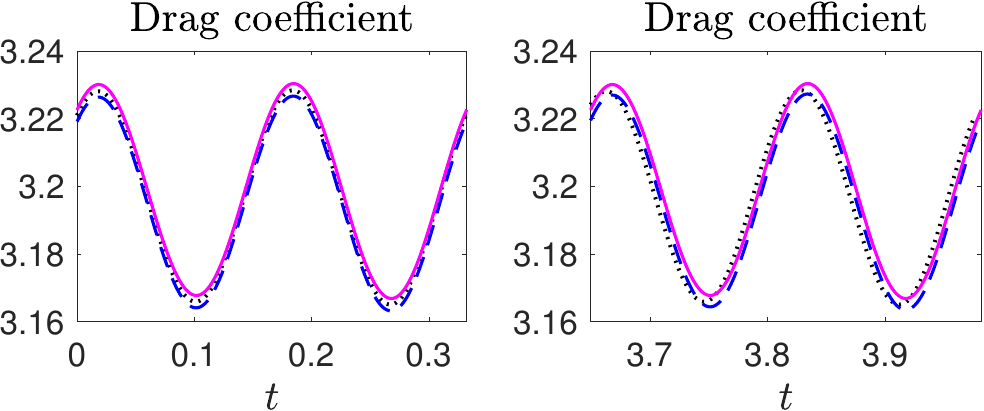}\hspace*{1em}
 \includegraphics[height=2.5truecm]{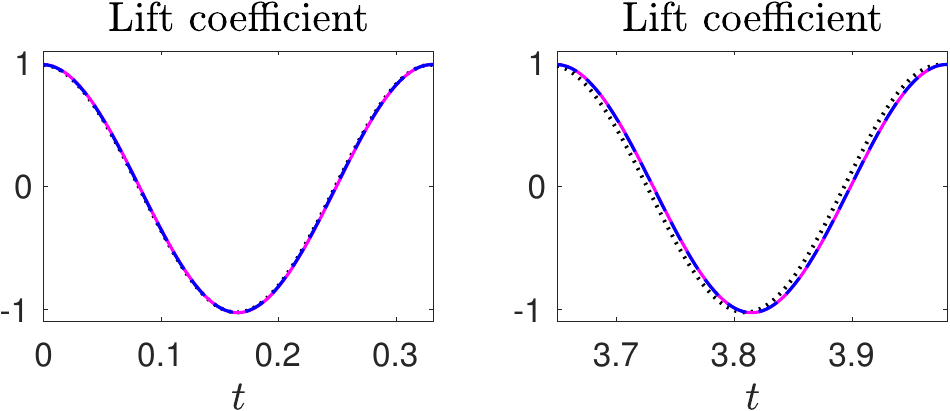}}
\caption{Drag (left) and lift (right) coefficients for the reference solution (black, {dotted line}), the FOM (magenta, {continuous line}), and the POD-ROM approximation with snapshots $\bu_h^j$ as data set and $L^2$ inner product {(blue, discontinuous line)}.
}\label{detalles} 
\end{figure}
In addition, the results for the lift coefficient corresponding to the FOM and the POD-ROM approximation are on top of each other (in fact the maximum relative error of this coefficient with respect to that of the FOM in the 12 periods is below $1.7\times 10^{-4}$). Some differences can be seen in the drag coefficient. However, the maximum relative error with respect the coefficient computed with the FOM  is below 0.0011.

The  summary of the results is that there are no significant differences in the POD-ROM approximation, i.e., 
they are quite independent of the data set used. Only some minor differences, {namely that} the projection errors 
depend on the data  set when the $H_0^1$ inner product is used, could be observed in the computation of the POD basis.

\section{Conclusions}
We analyzed reduced order models for the incompressible Na\-vier--Stokes equations based on proper orthogonal decomposition methods. The influence of including approximations to the time derivative in the set of snapshots
was studied. Our set of snapshots is
constituted by the approximation to the velocity at the initial time together with approximations to the time derivative at different times. The
approximation to the velocity at the initial time can be replaced by the mean value and 
then only approximations to the time derivatives are required
for applying the POD method to the fluctuations. The Galerkin time derivative can be replaced by any other approximation as the standard difference quotient. 

We studied the differences between projecting onto $L^2$ and ${H_0^1}$. We proved that including the Galerkin time derivatives (or the difference quotients) leads to pointwise estimates for both projections. In the $L^2$ case, 
error bounds can be proved in the no-time-derivatives case (with only snapshots for the velocity) as shown in the literature, e.g., see \cite{novo_rubino}. However, the time derivatives approach is
useful also in this case to get pointwise estimates. 
In  the numerical analysis, we utilized the projection of the continuous-in-time Galerkin approximation since this allows 
to use in practice any time integrator for computing the snapshots. It is easy to include the corresponding error in time in the
bounds of the present paper. Also, different times can be considered to compute the set of snapshots and the fully discrete POD approximations. Finally, as in one of the methods in \cite{novo_rubino}, we added grad-div stabilization to the approximations computed with the FOM and POD methods to be able to prove error bounds in which the constants do not depend on inverse powers of the viscosity.

In the numerical studies, we compared three different sets of snapshots, for both inner products, 
where the elements in the data sets were $\by_h^j=\bu_h^j-\overline \bu_h$, $\by_h^j=T\bu_{u,t}^j$, and $\by_h^j=T\delta_t\bu_h^j=T(\bu_h^j-\bu_h^{j-1})/\Delta t$. While the snapshots were computed using only one period,  we showed
that very good approximations are obtained to the lift and drag coefficients in a time interval of 12 periods. In our numerical studies there were no significant differences between the different procedures. 
More comprehensive studies have to further investigate this topic. 

Altogether, we think that the  rigorous numerical
analysis presented in this paper shows interesting properties and sharp bounds for the different methods. 
It also supports the
idea of the recent paper \cite{locke_singler} which shows that in the case of the heat equation
there is no need to include in the set of snapshots other than one approximation to the solution at a fixed time. We prove that for the approximation to the incompressible Navier--Stokes equations the same situation holds. Moreover, in case of applying the POD method to fluctuations, as it is standard, only snapshots of
the time derivative are needed.

\appendix\label{app}

\section{Robust bounds for the terms with the second temporal derivative}

In this section we derive bounds for the terms \eqref{second_dev_2} and \eqref{second_dev}. As a consequence, the terms on the right-hand sides of \eqref{eq:cota_finalSUPv} and \eqref{eq:error3_b_L2}
that contain  \eqref{second_dev} and  \eqref{second_dev_2}, respectively, are bounded. 

The constants in the bounds below do not blow up as $\nu\to 0$. It will be assumed that all functions are 
sufficiently smooth such that the performed operations are well defined.
{\begin{lemma}Let $\bu_h$ be the Galerkin approximation defined in \eqref{eq:gal_grad_div2} and
let $\bs_h^m$ be the modified Stokes projection defined in \eqref{stokespro_mod_def}. Let us denote by $
\be_h=\bu_h-\bs_h^m$. Assume for simplicity $\bu_h(0)=\bs_h^m(0)$. Then, the following bounds hold
\begin{eqnarray}\label{final_et}
\int_0^t\|\be_{h,t}(s)\|_0^2\ ds &\le& C h^{2(l-1)},\\
\label{eq:est_u_tt}
\int_0^t \|\be_{h,tt}(s)\|_0^2\ ds &\le& C h^{2(l-2)},\\
\label{seacabo}
\int_0^t \|\nabla(\be_{h,tt})(s)\|_0^2\ ds &\le& C h^{2(l-3)},
\end{eqnarray}
where the constant $C$ is independent of inverse powers of $\nu$.
\end{lemma}
}
\begin{proof}
 {We bound} 
\begin{equation}\label{eq:bu_h_tt_tria}
\|\bu_{h,tt}\|_0 \le \|\bu_{h,tt}-\bs^m_{h,tt}\|_0 + \|\bu_{tt}\|_0 + \|\bs^m_{h,tt}-\bu_{tt}\|_0 ,
\end{equation}
and likewise for the norm of the gradient. {Below, the first term on the right-hand side will be shown to be bounded.} 
The second term is bounded because the solution is sufficiently smooth, 
and then the last term can be bounded with \eqref{stokespro_mod}, where the 
notation $\bs^m_{h,tt}$ has to be understood that the modified Stokes projection is applied to $\bu_{tt}$. 

The error equation for {\eqref{eq:gal_grad_div2}} is given by 
\begin{eqnarray}\label{eq:gal_err}
\lefteqn{\hspace*{-8em}\left(\be_{h,t},\bvar_h\right)+\nu(\nabla \be_h,\nabla \bvar_h)+b_h(\bu_h,\bu_h,\bvar_h)-b_h(\bu,\bu,\bvar_h)+
\mu(\nabla \cdot\be_h,\nabla \cdot \bvar_h)}
\nonumber\\
&=&(\sigma_1,\bvar_h) +(\sigma_2,\nabla\cdot \bvar_h)
\quad\forall\ \bvar_h\in {\boldsymbol V}_h^l,
\end{eqnarray}
with 
$\sigma_1=\bu_t-\bs^m_{h,t}$, $\sigma_2=(p-P_{Q_h}p)+\mu(\nabla \cdot (\bu-\bs_h^m)$,  
and $P_{Q_h}p$ is the best approximation of $p$ in $Q_h^l$. Note that if a (discrete) velocity function is
in $\boldsymbol V$ (or ${\boldsymbol V}_h^l$), then also the temporal derivatives of this function are in 
the same space. Hence, taking $\bvar_h=\be_{h,t} \in {\boldsymbol V}_h^l$ in \eqref{eq:gal_err} 
and using the inverse inequality \eqref{inv} yields
\begin{eqnarray}\label{eq:app_00}
\lefteqn{\|\be_{h,t}\|_0^2+\frac{d}{dt}\frac{1}{2}\nu\|\nabla \be_h\|_0^2+\frac{d}{dt}\frac{1}{2}\mu\|\nabla \cdot  \be_h\|_0^2 \le|b_h(\bu_h-\bu,\bu,\be_{h,t})| } \nonumber\\
&&
+|b_h(\bu_h,\bu_h-\bu,\be_{h,t})| +\|\sigma_1\|_0\|\be_{h,t}\|_0+c_{\rm inv}\|\sigma_2\|_0 h^{-1}\|\be_{h,t}\|_0.
\end{eqnarray}
The first term on the right-hand side is estimated with H\"older's inequality, \eqref{eq:cota_grad_div}, and 
Young's inequality
\begin{eqnarray*}
|b_h(\bu_h-\bu,\bu,\be_{h,t})|&\le& \|\bu_h-\bu\|_0\|\nabla \bu\|_\infty\|\be_{h,t}\|_0
+\frac{1}{2}\|\nabla \cdot (\bu_h-\bu)\|_0\|\bu\|_\infty\|\be_{h,t}\|_0\\
&\le& C h^{2(l-1)}+\frac{1}{8}\|\be_{h,t}\|_0^2.
\end{eqnarray*}
For the second term, in addition the following estimate from \cite[Lemma~6.11]{John}
{
\[
((\bu\cdot\nabla) \bv,\bw)\le C\|\bu\|_1\|\nabla \bv\|_0\|\bw\|_1,\quad \bu,\bv,\bw \in H^1(\Omega)^d,
\]
}
the inverse inequality \eqref{inv}, and the condition $l\ge 2$ are 
utilized
\begin{eqnarray*}
|b_h(\bu_h,\bu_h-\bu,\be_{h,t})|&\le& |b_h(\bu_h-\bu,\bu_h-\bu,\be_{h,t})|+|b_h(\bu,\bu_h-\bu,\be_{h,t})|\\
&\le& C \|\bu_h-\bu\|_1^2\|\be_{h,t}\|_1+\|\bu\|_\infty\|\nabla (\bu_h-\bu)\|_0\|\be_{h,t}\|_0\\
&\le & C \|\bu_h-\bu\|_1^2 c_{\rm inv} h^{-1}\|\be_{h,t}\|_0+C \|\bu_h-\bu\|_1\|\be_{h,t}\|_0\\
&\le & C \left(h^{4l-6}+h^{2l-2}\right)+\frac{1}{8}\|\be_{h,t}\|_0^2\le C h^{2(l-1)}+\frac{1}{8}\|\be_{h,t}\|_0^2.
\end{eqnarray*}
For the last two terms, we obtain, with \eqref{stokespro_mod}, the $L^2(\Omega)$ best approximation 
error for the pressure, and Young's inequality, the following bound 
\begin{eqnarray*}
\|\sigma_1\|_0\|\be_{h,t}\|_0+c_{\rm inv}\|\sigma_2\|_0 h^{-1}\|\be_{h,t}\|_0
&\le& 2\|\sigma_1\|_0^2+c_{\rm inv}^2\|\sigma_2\|_0^2 h^{-2}+\frac{1}{4}\|\be_{h,t}\|_0^2\\
&\le&C h^{2(l-1)}+\frac{1}{4}\|\be_{h,t}\|_0^2.
\end{eqnarray*}
Absorbing terms in the left-hand side of \eqref{eq:app_00} and collecting the other terms on the right-hand 
side leads to 
\begin{equation}\label{eq:app_02}
\frac{1}{2}\left(\|\be_{h,t}\|_0^2+\frac{d}{dt}\nu\|\nabla \be_h\|_0^2+\frac{d}{dt}\mu\|\nabla \cdot  \be_h\|_0^2\right)
\le C h^{2(l-1)},
\end{equation}
{from which, taking into account $\be_h(0)=0$}, we derive the estimate
\[
\int_0^t\|\be_{h,t}(s)\|_0^2\ ds +\nu\|\nabla  \be_h(t)\|_0^2+\mu\|\nabla \cdot \be_h(t)\|_0^2\le C h^{2(l-1)}.
\]
so that \eqref{final_et} holds. 

Considering $t\to0$ in \eqref{eq:gal_err}, the viscous term and the grad-div stabilization term 
vanish, since $\be_h(0)=\boldsymbol 0$. {Following \cite{HR82} the initial pressure $p(0)$ is the 
solution of the problem
\begin{eqnarray*}
\Delta p(0)&=&\nabla \cdot \bff (0)-\nabla \cdot ((\bu^0\cdot \nabla) \bu^0) \quad \mbox{in } \Omega,\\
\frac{\partial p(0)}{\partial \bn}&=&(\nu \Delta \bu^0+\bff (0)-(\bu^0\cdot \nabla) \bu^0)\cdot \bn\quad \mbox{on }\partial \Omega,
\end{eqnarray*}
where $\bn$ is the outward pointing unit normal vector The above problem defines a unique pressure up to a constant.}
{Then, one can repeat the same analysis as above at time $t=0$ starting from \eqref{eq:app_00} (with only the
term $\|\be_{h,t}(0)\|_0^2$ on the left-hand side)}
and obtains, instead of \eqref{eq:app_02},
\begin{equation*}
\frac12 \|\be_{h,t}(0)\|_0^2 \le C h^{2(l-1)}.
\end{equation*}
Applying the inverse inequality yields
\begin{equation}\label{eq:app_03}
\|(\nabla\be_{h,t})(0)\|_0^2 \le C h^{2(l-2)}, \quad \|(\nabla\cdot \be_{h,t})(0)\|_0^2 \le C h^{2(l-2)},
\end{equation}
such that the norms on the left-hand sides are bounded for $l\ge 2$. 

Taking the time derivative of \eqref{eq:gal_err} gives
\begin{eqnarray}\label{eq:gal_err_t}
\lefteqn{
\left(\be_{h,tt},\bvar_h\right)+\nu(\nabla \be_{h,t},\nabla \bvar_h)+b_h(\bu_{h,t},\bu_h,\bvar_h)-b_h(\bu_{t},\bu,\bvar_h)}
\nonumber\\
&&+b_h(\bu_h,\bu_{h,t},\bvar_h)-b_h(\bu,\bu_{t},\bvar_h)+
\mu(\nabla \cdot\be_{h,t},\nabla \cdot \bvar_h)
\\
&=&(\sigma_{1,t},\bvar_h) +(\sigma_{2,t},\nabla\cdot \bvar_h)
\quad\forall\ \bvar_h\in {\boldsymbol V}_h^l,\nonumber
\end{eqnarray}
Choosing $\bvar_h=\be_{h,tt}$ and arguing as before leads to 
\begin{eqnarray}\label{eq:app_01}
\lefteqn{\|\be_{h,tt}\|_0^2+\frac{d}{dt}\frac{1}{2}\nu\|\nabla \be_{h,t}\|_0^2+\frac{d}{dt}\frac{1}{2}\mu\|\nabla \cdot  \be_{h,t}\|_0^2
}\nonumber\\
&\le &  \|\sigma_{1,t}\|_0^2+c_{\rm inv}^2h^{-2}\|\sigma_{2,t}\|_0^2
+\frac{1}{4}\|\be_{h,tt}\|_0^2+|b_h(\bu_{h,t}-\bu_{t},\bu,\be_{h,tt})|\nonumber\\
&&+|b_h(\bu_{h,t},\bu-\bu_h,\be_{h,tt})| +|b_h(\bu_h-\bu,\bu_{t},\be_{h,tt})|+|b_h(\bu_h,\bu_{t}-\bu_{h,t},\be_{h,tt})|. 
\end{eqnarray}
Let us observe that the bounds for $\|\sigma_{1,t}\|_0$ and $\|\sigma_{2,t}\|_0$ depend on the regularity 
of $\bu_{t}$, $\bu_{tt}$, and $p_t$, but all are bounded if the solution is sufficiently regular.
For the nonlinear terms, the same arguments as above are used. 
For the first term, we obtain 
\begin{eqnarray}\label{posiaca}
\lefteqn{
|b_h(\bu_{h,t}-\bu_{t},\bu,\be_{h,tt})|}\nonumber\\
&\le &\|\bu_{h,t}-\bu_{t}\|_0\|\nabla \bu\|_\infty\|\be_{h,tt}\|_0
+\frac{1}{2}\|\nabla \cdot (\bu_{h,t}-\bu_{t})\|_0\|\bu\|_\infty\|\be_{h,tt}\|_0\nonumber\\
&\le& C \|\bu_{t}-\bu_{h,t}\|_1^2+\frac{1}{16}\|\be_{h,tt}\|_0^2.
\end{eqnarray}
It follows from \eqref{stokespro_mod},  \eqref{final_et}, and the inverse inequality \eqref{inv} that 
\[
\int_0^t  \|\bu_{t}-\bu_{h,t}\|_1^2\le C h^{2(l-2)},
\]
so that absorbing the second term on the right-hand side of \eqref{posiaca} in  the left-hand side 
of \eqref{eq:app_01} and integrating in time, the corresponding 
first term on the right-hand side of \eqref{posiaca}  is bounded.
For the second nonlinear term on the right-hand side of \eqref{eq:app_01}, we obtain
\begin{eqnarray*}
|b_h(\bu_{h,t},\bu-\bu_h,\be_{h,tt})|&\le& C\|\bu_{h,t}\|_1\|\bu-\bu_h\|_1c_{\rm inv}h^{-1}\|\be_{h,tt}\|_0
\\&\le& C h^{2(l-2)}\|\bu_{h,t}\|_1^2+\frac{1}{16}\|\be_{h,tt}\|_0^2.
\end{eqnarray*}
Again, the last term on the right-hand side can be absorbed in the left-hand side of \eqref{eq:app_01} and the integral with 
respect to time of the first term is bounded. 
For the third nonlinear term we argue as for the second one to get 
\[
|b_h(\bu_h-\bu,\bu_{t},\be_{h,tt})|\le \|\bu_h-\bu\|_1 \|\bu_{t}\|_1 \|\be_{h,tt}\|_1
\le C h^{2(l-2)}\|\bu_{t}\|_1^2+\frac{1}{16}\|\be_{h,tt}\|_0^2.
\]
Using for the fourth nonlinear term the inverse inequality \eqref{inv} and \eqref{stokespro_mod}
gives
\begin{eqnarray*}
\lefteqn{
|b_h(\bu_h,\bu_{t}-\bu_{h,t},\be_{h,tt})| \le \|\bu_h\|_\infty \|\bu_{t}-\bu_{h,t}\|_1\|\be_{h,tt}\|_0 }\\
&& 
+\frac{1}{2}\|\nabla \cdot (\bu_h-\bs_h)\|_\infty\|\bu_{t}-\bu_{h,t}\|_0\|\be_{h,tt}\|_0
+\frac{1}{2}\|\nabla \cdot \bs_h\|_\infty\|\bu_{t}-\bu_{h,t}\|_0\|\be_{h,tt}\|_0\\
&\le& C \left(\|\bu_{t}-\bu_{h,t}\|_1^2+c_{\rm inv}^2h^{-d} h^{2(l-1)}\|\bu_{t}-\bu_{h,t}\|_0^2\right)+\frac{1}{16}\|\be_{h,tt}\|_0^2.
\end{eqnarray*}
Again, the last term on the right-hand side above is absorbed into the left-hand side of \eqref{eq:app_01},
the integral of the first term is bounded, and the integral for the second is
\[
c_{\rm inv}^2h^{-d} h^{2(l-1)}\int_0^t\|\bu_{t}-\bu_{h,t}(s)\|_0^2\ ds \le C h^{4l-4-d},
\]
which is bounded since $l\ge2$. Collecting all error bounds and taking \eqref{eq:app_03} into account, we conclude 
\eqref{eq:est_u_tt} where the constant does not depend explicitly on inverse powers of $\nu$. The inverse inequality gives \eqref{seacabo}.
\end{proof}

Estimate \eqref{eq:est_u_tt} can be applied, in combination with \eqref{eq:bu_h_tt_tria}, in \eqref{eq:error3_b_L2}. To bound \eqref{second_dev}, which is the term appearing in \eqref{eq:cota_finalSUPv}, we apply \eqref{seacabo}
such that a robust estimate is proved only for pairs of finite element spaces with $l\ge3$.

\bibliographystyle{siamplain}
\bibliography{references}

\begin{thebibliography}{10}

\bibitem{Adams}
{\sc R.~A. Adams}, {\em Sobolev spaces}, Academic Press [A subsidiary of
  Harcourt Brace Jovanovich, Publishers], New York-London, 1975.
\newblock Pure and Applied Mathematics, Vol. 65.

\bibitem{BF}
{\sc F.~Brezzi and R.~S. Falk}, {\em Stability of higher-order {H}ood-{T}aylor
  methods}, SIAM J. Numer. Anal., 28 (1991), pp.~581--590,
  \url{https://doi.org/10.1137/0728032}.

\bibitem{John_et_al_vp}
{\sc A.~Caiazzo, T.~Iliescu, V.~John, and S.~Schyschlowa}, {\em A numerical
  investigation of velocity-pressure reduced order models for incompressible
  flows}, J. Comput. Phys., 259 (2014), pp.~598--616,
  \url{https://doi.org/10.1016/j.jcp.2013.12.004}.

\bibitem{chapelle}
{\sc D.~Chapelle, A.~Gariah, and J.~Sainte-Marie}, {\em Galerkin approximation
  with proper orthogonal decomposition: new error estimates and illustrative
  examples}, ESAIM Math. Model. Numer. Anal., 46 (2012), pp.~731--757,
  \url{https://doi.org/10.1051/m2an/2011053}.

\bibitem{chenSiam}
{\sc H.~Chen}, {\em Pointwise error estimates for finite element solutions of
  the {S}tokes problem}, SIAM J. Numer. Anal., 44 (2006), pp.~1--28,
  \url{https://doi.org/10.1137/S0036142903438100}.

\bibitem{periodic_orbit}
{\sc J.-H. Chen, W.~G. Pritchard, and S.~J. Tavener}, {\em Bifurcation for flow
  past a cylinder between parallel planes}, J. Fluid Mech., 284 (1995),
  pp.~23--41, \url{https://doi.org/10.1017/S0022112095000255},
  \url{https://doi.org/10.1017/S0022112095000255}.

\bibitem{Cia78}
{\sc P.~G. Ciarlet}, {\em The finite element method for elliptic problems},
  vol.~40 of Classics in Applied Mathematics, Society for Industrial and
  Applied Mathematics (SIAM), Philadelphia, PA, 2002,
  \url{https://doi.org/10.1137/1.9780898719208}.
\newblock Reprint of the 1978 original [North-Holland, Amsterdam].

\bibitem{grad-div1}
{\sc J.~de~Frutos, B.~Garc\'{\i}a-Archilla, V.~John, and J.~Novo}, {\em
  Grad-div stabilization for the evolutionary {O}seen problem with inf-sup
  stable finite elements}, J. Sci. Comput., 66 (2016), pp.~991--1024,
  \url{https://doi.org/10.1007/s10915-015-0052-1}.

\bibitem{NS_grad_div}
{\sc J.~de~Frutos, B.~Garc\'{\i}a-Archilla, V.~John, and J.~Novo}, {\em
  Analysis of the grad-div stabilization for the time-dependent
  {N}avier-{S}tokes equations with inf-sup stable finite elements}, Adv.
  Comput. Math., 44 (2018), pp.~195--225,
  \url{https://doi.org/10.1007/s10444-017-9540-1}.

\bibitem{Evans}
{\sc L.~C. Evans}, {\em Partial differential equations}, vol.~19 of Graduate
  Studies in Mathematics, American Mathematical Society, Providence, RI,
  second~ed., 2010, \url{https://doi.org/10.1090/gsm/019},
  \url{https://doi.org/10.1090/gsm/019}.

\bibitem{GJN21}
{\sc B.~Garc\'{\i}a-Archilla, V.~John, and J.~Novo}, {\em On the convergence
  order of the finite element error in the kinetic energy for high {R}eynolds
  number incompressible flows}, Comput. Methods Appl. Mech. Engrg., 385 (2021),
  pp.~Paper No. 114032, 54, \url{https://doi.org/10.1016/j.cma.2021.114032}.

\bibitem{pod_da_nos}
{\sc B.~Garc\'{\i}a-Archilla, J.~Novo, and S.~Rubino}, {\em Error analysis of
  proper orthogonal decomposition data assimilation schemes with grad-div
  stabilization for the {N}avier-{S}tokes equations}, J. Comput. Appl. Math.,
  411 (2022), pp.~Paper No. 114246, 30,
  \url{https://doi.org/10.1016/j.cam.2022.114246}.

\bibitem{bosco_titi_yo}
{\sc B.~Garc\'{\i}a-Archilla, J.~Novo, and E.~S. Titi}, {\em Uniform in time
  error estimates for a finite element method applied to a downscaling data
  assimilation algorithm for the {N}avier-{S}tokes equations}, SIAM J. Numer.
  Anal., 58 (2020), pp.~410--429, \url{https://doi.org/10.1137/19M1246845}.

\bibitem{BDF2}
{\sc B.~García-Archilla and J.~Novo}, {\em {Robust error bounds for the
  Navier–Stokes equations using implicit-explicit second-order BDF method
  with variable steps}}, IMA Journal of Numerical Analysis,  (2022),
  \url{https://doi.org/10.1093/imanum/drac058}.

\bibitem{HR82}
{\sc J.~G. Heywood and R.~Rannacher}, {\em Finite element approximation of the
  nonstationary {N}avier-{S}tokes problem. {I}. {R}egularity of solutions and
  second-order error estimates for spatial discretization}, SIAM J. Numer.
  Anal., 19 (1982), pp.~275--311, \url{http://dx.doi.org/10.1137/0719018}.

\bibitem{hey_ran_IV}
{\sc J.~G. Heywood and R.~Rannacher}, {\em Finite-element approximation of the
  nonstationary {N}avier-{S}tokes problem. {IV}. {E}rror analysis for
  second-order time discretization}, SIAM J. Numer. Anal., 27 (1990),
  pp.~353--384, \url{https://doi.org/10.1137/0727022}.

\bibitem{iliescu_et_al_q}
{\sc T.~Iliescu and Z.~Wang}, {\em Are the snapshot difference quotients needed
  in the proper orthogonal decomposition?}, SIAM J. Sci. Comput., 36 (2014),
  pp.~A1221--A1250, \url{https://doi.org/10.1137/130925141}.

\bibitem{John}
{\sc V.~John}, {\em Finite element methods for incompressible flow problems},
  vol.~51 of Springer Series in Computational Mathematics, Springer, Cham,
  2016, \url{https://doi.org/10.1007/978-3-319-45750-5}.

\bibitem{nos_pos_supg}
{\sc V.~John, B.~Moreau, and J.~Novo}, {\em Error analysis of a
  {SUPG}-stabilized {POD}-{ROM} method for convection-diffusion-reaction
  equations}, Comput. Math. Appl., 122 (2022), pp.~48--60,
  \url{https://doi.org/10.1016/j.camwa.2022.07.017},
  \url{https://doi.org/10.1016/j.camwa.2022.07.017}.

\bibitem{kean_sch}
{\sc K.~Kean and M.~Schneier}, {\em Error analysis of supremizer pressure
  recovery for {POD} based reduced-order models of the time-dependent
  {N}avier-{S}tokes equations}, SIAM J. Numer. Anal., 58 (2020),
  pp.~2235--2264, \url{https://doi.org/10.1137/19M128702X}.

\bibitem{koc_rubino_et_al}
{\sc B.~Koc, S.~Rubino, M.~Schneier, J.~Singler, and T.~Iliescu}, {\em On
  optimal pointwise in time error bounds and difference quotients for the
  proper orthogonal decomposition}, SIAM J. Numer. Anal., 59 (2021),
  pp.~2163--2196, \url{https://doi.org/10.1137/20M1371798}.

\bibitem{kostova}
{\sc T.~Kostova-Vassilevska and G.~M. Oxberry}, {\em Model reduction of
  dynamical systems by proper orthogonal decomposition: error bounds and
  comparison of methods using snapshots from the solution and the time
  derivatives}, J. Comput. Appl. Math., 330 (2018), pp.~553--573,
  \url{https://doi.org/10.1016/j.cam.2017.09.001}.

\bibitem{kunisch}
{\sc K.~Kunisch and S.~Volkwein}, {\em Galerkin proper orthogonal decomposition
  methods for parabolic problems}, Numer. Math., 90 (2001), pp.~117--148,
  \url{https://doi.org/10.1007/s002110100282}.

\bibitem{KV_2002}
{\sc K.~Kunisch and S.~Volkwein}, {\em Galerkin proper orthogonal decomposition
  methods for a general equation in fluid dynamics}, SIAM J. Numer. Anal., 40
  (2002), pp.~492--515, \url{https://doi.org/10.1137/S0036142900382612},
  \url{https://doi.org/10.1137/S0036142900382612}.

\bibitem{locke_singler}
{\sc S.~K. Locke and J.~R. Singler}, {\em A new approach to proper orthogonal
  decomposition with difference quotiens}, arXiv:2106.10224v1 [math.NA] 18 Jun
  2021.

\bibitem{novo_rubino}
{\sc J.~Novo and S.~Rubino}, {\em Error analysis of proper orthogonal
  decomposition stabilized methods for incompressible flows}, SIAM J. Numer.
  Anal., 59 (2021), pp.~334--369, \url{https://doi.org/10.1137/20M1341866}.

\bibitem{periodic}
{\sc J.~S\'{a}nchez, M.~Net, B.~Garc\'{\i}a-Archilla, and C.~Sim\'{o}}, {\em
  Newton-{K}rylov continuation of periodic orbits for {N}avier-{S}tokes flows},
  J. Comput. Phys., 201 (2004), pp.~13--33,
  \url{https://doi.org/10.1016/j.jcp.2004.04.018}.

\bibitem{bench}
{\sc M.~Sch{\"a}fer, S.~Turek, F.~Durst, E.~Krause, and R.~Rannacher}, {\em
  Benchmark Computations of Laminar Flow Around a Cylinder}, Vieweg+Teubner
  Verlag, Wiesbaden, 1996, pp.~547--566,
  \url{https://doi.org/10.1007/978-3-322-89849-4_39}.

\bibitem{SL17}
{\sc P.~W. Schroeder and G.~Lube}, {\em Pressure-robust analysis of
  divergence-free and conforming {FEM} for evolutionary incompressible
  {N}avier-{S}tokes flows}, J. Numer. Math., 25 (2017), pp.~249--276,
  \url{https://doi.org/10.1515/jnma-2016-1101}.

\bibitem{singler}
{\sc J.~R. Singler}, {\em New {POD} error expressions, error bounds, and
  asymptotic results for reduced order models of parabolic {PDE}s}, SIAM J.
  Numer. Anal., 52 (2014), pp.~852--876,
  \url{https://doi.org/10.1137/120886947}.

\bibitem{hood0}
{\sc C.~Taylor and P.~Hood}, {\em A numerical solution of the {N}avier-{S}tokes
  equations using the finite element technique}, Internat. J. Comput. \&
  Fluids, 1 (1973), pp.~73--100,
  \url{https://doi.org/10.1016/0045-7930(73)90027-3}.

\bibitem{Zha05}
{\sc S.~Zhang}, {\em A new family of stable mixed finite elements for the 3{D}
  {S}tokes equations}, Math. Comp., 74 (2005), pp.~543--554,
  \url{http://dx.doi.org/10.1090/S0025-5718-04-01711-9}.

\end{thebibliography}
\end{document}